\RequirePackage{fix-cm}
\documentclass[smallextended,envcountsame,numbook]{svjour3}      % onecolumn (second format)
\smartqed  % flush right qed marks, e.g. at end of proof
\usepackage{graphicx}
 \usepackage{hyperref}
 \hypersetup{
     colorlinks=true,
     linkcolor=blue,
     filecolor=magenta,      
     urlcolor=cyan,
     citecolor=red,
}
\usepackage{amsfonts, amssymb, amsbsy, amsmath}
 \usepackage{blkarray}
 \usepackage{graphicx}
 \usepackage{float}
 \usepackage{epstopdf}
 \usepackage{algorithmic}
 \ifpdf
   \DeclareGraphicsExtensions{.eps,.pdf,.png,.jpg}
 \else
   \DeclareGraphicsExtensions{.eps}
 \fi
 \usepackage{mathrsfs}
 \usepackage{amscd,empheq} %,yhmath}
 \usepackage{yfonts,bm}

 \usepackage[pdf]{graphviz}

 \usepackage[all,2cell]{xy}
 \usepackage{pb-diagram}
 \usepackage{pb-xy}

 \usepackage{tikz}
 \usepackage{tikz-cd}
 \usetikzlibrary{patterns}
 \usetikzlibrary{positioning}
 \usetikzlibrary{shapes,decorations.markings,arrows.meta}
 \usetikzlibrary{matrix,arrows}
 \tikzset{hfit/.style={rounded rectangle, inner xsep=0pt},
           vfit/.style={rounded corners}}

 \usepackage{morewrites}
 \usepackage{scalerel}
 \usepackage{mathabx}
 \usepackage{ascii}
 \usepackage{palatino}
 
\usepackage{enumitem}
\setlist[enumerate]{leftmargin=.5in}
\setlist[itemize]{leftmargin=.5in}

\usepackage{amsopn}

\newcommand{\dff}{{\mathrm d}}

\newcommand{\sCo}{{\mathsf{Co}}}

\newcommand{\spi}{\mathsf{p}}

\newcommand{\IIonei}{\bfQ_{\mathbf 1}}
\newcommand{\IIi}{\bfQ_{\mathbf 2}}

\newcommand{\IIni}{\bfQ_{\mathbf n}}
\newcommand{\IIthreei}{\bfQ_{\mathbf 3}}

\newcommand{\smin}{\smallsetminus}

\renewcommand{\subset}{\subseteq}

\newcommand{\grade}{\mathrm{grd}}
\newcommand{\flt}{\mathrm{flt}}

\definecolor{Blue}{RGB}{0,162,255}
\definecolor{Orange}{RGB}{243,144,25}
\definecolor{Red}{RGB}{236,93,87}

\newcommand{\F}{{\mathbb{K}}}

\newcommand{\Z}{{\mathbb{Z}}}

\newcommand{\sE}{{\mathsf E}}
\newcommand{\sH}{{\mathsf H}}

\newcommand{\sO}{{\mathsf O}}

\newcommand{\sP}{{\mathsf P}}
\newcommand{\sQ}{{\mathsf Q}}

\newcommand{\sh}{{\mathsf h}}
\newcommand{\sk}{{\mathsf k}}

\renewcommand{\imath}{{\mathsf i}}
\renewcommand{\jmath}{{\mathsf j}}

\newcommand{\id}{\mathrm{id}}

\DeclareMathOperator{\img}{im}

\DeclareMathOperator{\rank}{rank}

\newcommand{\image}{\mathrm{im~}}

\newcommand{\Sub}{\mathsf{Sub}}
\newcommand{\mPi}{\mathrm{\Pi}}

\newcommand{\sRmod}{\mathbf{Ab}}

\newcommand{\sAb}{\mathbf{Ab}}
\newcommand{\sAbd}{\mathbf{Ab}_\dff}
\newcommand{\bfE}{\text{\bf E}}

\newcommand{\bfB}{\text{\bf B}}

\newcommand{\bfQ}{\text{\bf Q}}

{\begin{snugshade}\begin{quote}}
{\hfill\end{quote}\end{snugshade}}

\definecolor{shadecolor}{rgb}{0.8,0.8,0.8}

\begin{document}
\begin{sloppypar}

\title{Graded differential groups, Cartan-Eilenberg systems\\
 and conjectures in Conley index theory%\thanks{Grants or other notes
}
%\subtitle{}
%\\ If so, write it here}

\titlerunning{Graded differential groups \& Cartan-Eilenberg systems}        % if too long for running head

\author{Kelly Spendlove        \and
        Robert Vandervorst %etc.
}


\institute{K. Spendlove \at
              Mathematical Institute \\
              University of Oxford\\
              Oxford UK\\
              \email{spendlove@maths.ox.ac.uk}           %  \\
           \and
           R.C.A.M. Vandervorst \at
              Department of Mathematics\\
              VU University Amsterdam\\
              The Netherlands\\
              \email{r.c.a.m.vander.vorst@vu.nl}
}

\date{Version date: \today}

\maketitle

\begin{abstract}
Cartan-Eilenberg systems play an prominent role in the homological algebra of filtered and graded differential groups and (co)chain complexes in particular.
We define the concept of Cartan-Eilenberg systems of abelian groups
over a poset. Our main result states that a  filtered chain isomorphism between free, $\sP$-graded differential groups is equivalent to an isomorphism between associated Cartan-Eilenberg systems. An application of this result to the theory of dynamical systems addresses two open conjectures posed by J. Robbin and D. Salamon regarding  uniqueness type questions for connection matrices, as detailed in \cite{robbin:salamon2} and  \cite{atm}.
The main result of this paper also proves that the connection matrix theories in \cite{fran},
\cite{robbin:salamon2} and \cite{hms} are equivalent in the setting of vector spaces, as well as uniqueness of connection matrices for Morse-Smale gradient systems, cf.\ \cite{Reineck}.

\end{abstract}

%\begin{sloppypar}
%\input{outline}

\section{Introduction}
\label{intronew}
In the realm of algebraic topology and homological algebra, differential groups graded by a poset $\sP$ play a crucial role 
in the theory of spectral sequences and spectral systems
and their applications in various fields, in particular dynamical systems. This paper studies the intricate relationship between  $\sP$-graded differential groups and Cartan-Eilenberg systems, cf.\ \cite{CE}.
Our investigation revisits the traditional concept of Cartan-Eilenberg systems 
by considering
 abelian groups over an arbitrary poset.
In particular, an application of our main result to  dynamical systems theory 
confirms
two open conjectures posited by J. Robbin and D. Salamon, referenced in \cite{robbin:salamon2} and \cite{atm}, which concern the uniqueness of connection matrices—a pivotal concept in the analysis of the global behavior of dynamical systems. 

In order to state the main results in this paper we start with a brief, intuitive definition of Cartan-Eilenberg systems over a lattice. Let $\sP$ be a finite poset and let $\sO(\sP)$ be the lattice of down-sets in $\sP$.
A Cartan-Eilenberg system $\bfE$ over $\sO(\sP)$ consists of a functor $\sE$ from the arrow category of $\sO(\sP)$, which consists of ordered pairs $(\alpha,\beta)$, with $\alpha\subset\beta$, to the category of abelian groups and an additional structure that accounts for exact triangles in homology, cf.\ Sect.\ \ref{catdefn} for a detailed definition.  
If we unpack the above informal definition we obtain the application
$(\alpha,\beta) \mapsto \sE(\alpha,\beta)=E^\beta_\alpha$.  
The functor $\sE$ yields the homomorphisms    $\ell\colon E^\beta_\alpha \to E^{\beta'}_{\alpha'}$ for all $(\alpha,\beta)\le (\alpha',\beta')$,\footnote{The order on order pairs is point wise, i.e. $(\alpha,\beta)\le (\alpha',\beta')$ if and only $\alpha\subset\alpha'$ and $\beta\subset\beta'$.} and
the composition $E^\beta_\alpha \xrightarrow[]{\ell} E^{\beta'}_{\alpha'}\xrightarrow[]{\ell} E^{\beta''}_{\alpha''}$ is given by $\ell\colon E^\beta_\alpha \to E^{\beta''}_{\alpha''}$ due to the transitivity in $\sO(\sP)$. Moreover there exist connecting homomorphisms $k\colon E^\gamma_\beta\to E^\beta_\alpha$ and
$k\colon E^{\gamma'}_{\beta'}\to E^{\beta'}_{\alpha'}$ such that the diagrams\footnote{In the special cases $(\alpha,\beta)\le (\alpha,\gamma)$ and $(\alpha,\gamma)\le (\beta,\gamma)$ the morphisms $\ell$ are denoted by $i$ and $j$ respectively.} 
\begin{equation}
\label{unpackeddiag12}
\begin{tikzcd}[column sep=small]
E^\beta_\alpha \arrow[rr, "i"] &                   & E^\gamma_\alpha \arrow[ld, "j"] \\
                  & E^\gamma_\beta \arrow[lu, "k"] &                  
\end{tikzcd}
\qquad\qquad
\begin{tikzcd}
E^\gamma_\beta \arrow[r, "k"] \arrow[d, "\ell"'] & E^\beta_\alpha \arrow[d, "\ell"] \\
E^{\gamma'}_{\beta'} \arrow[r, "k"]                 & E^{\beta'}_{\alpha'}               
\end{tikzcd}
\end{equation}
are exact and commutative respectively for all pairs of ordered triples $(\alpha,\beta,\gamma)\le (\alpha',\beta',\gamma')$.\footnote{An ordered triple $(\alpha,\beta,\gamma)$ is defined by $\alpha\subset\beta\subset \gamma$. The order on ordered triples is again point wise.}
By construction $\ell\colon E^\beta_\alpha \xrightarrow[\id]{\ell} E^\beta_\alpha$ is the identity homomorphism and the exactness of \eqref{unpackeddiag12}[left] shows that $E^\alpha_\alpha =0$ for all $\alpha \in \sO(\sP)$, cf.\ \cite{HelleRognes}.
A Cartan-Eilenberg system is \emph{excisive} if $E_\alpha^\beta\cong E_{\alpha'}^{\beta'}$ for all $\beta\smin\alpha=\beta'\smin\alpha'$, cf.\ Lem.\ \ref{execiso}.
A morphism between Cartan-Eilenberg systems $\bfE$ and $\bfE'$ is given by homomorphisms  $h^{\beta}_{\alpha}\colon E^\beta_\alpha \to E'^\beta_\alpha$ which 
satisfy the commutative diagram:
\begin{equation}
\label{dia:exacttriangles}
    \begin{tikzcd}[row sep=large, column sep=large]
{} \arrow[r] & E_\alpha^\beta \arrow[r, "i"] \arrow[d, "h^\beta_\alpha"] & E^\gamma_\alpha \arrow[r, "j"] \arrow[d, "h^{\gamma}_{\alpha}"] & E^\gamma_\beta \arrow[r, "k"] \arrow[d, "h^{\gamma}_{\beta}"] & E_\alpha^\beta \arrow[r] \arrow[d, "h^{\beta}_{\alpha}"] & {} \\
{} \arrow[r] & E'^\beta_\alpha \arrow[r, "i'"]                           & E'^\gamma_\alpha \arrow[r, "j'"]                                & E'^\gamma_\beta \arrow[r, "k'"]                               & E'^\beta_\alpha \arrow[r]                                & {}
\end{tikzcd}
\end{equation}
 any $\alpha\subset\beta\subset\gamma$. In this case we write $h\colon \bfE\to \bfE'$ as a morphism of Cartan-Eilenberg systems. 

A free, $\sP$-graded differential group is a free abelian group $C$ with an endomorphism $\dff_C$, with $\dff_C^2=0$, and a $\sP$-grading which yields  a decomposition
\begin{equation}
    \label{Pgrdecomp}
    C = \bigoplus_{p\in \sP} G_pC,
\end{equation}
where $G_pC\subset C$ are free subgroups. 
A free, $\sP$-graded differential group is strict if the restrictions $\dff\colon G_pC\to G_pC$ are trivial for all $p\in \sP$.
A homomorphism of differential groups is called a chain homomorphism. A chain homomorphism is $\sO(\sP)$-filtered if it preserves the filterings.
Precise definitions are provided in Section \ref{diffgroup} and \ref{precdiffgr}. 
As detailed in Section \ref{CEfordffgr}, a $\sP$-graded differential group $(C,\dff)$ gives rise to a Cartan-Eilenberg system $\bfE(C,\dff_C)$ with
$E$-terms $E_\alpha^\beta=H(G_{\beta\smin\alpha}C)$, where $(G_{\beta\smin\alpha}C,\dff_C)$ is a well-defined differential subgroup 
and $E_\alpha^\beta$ is its homology.
The main result of this paper can be phrased as follows:
\vspace{0.2cm}

\noindent {\bf Theorem A} (cf.\ Thm.\ \ref{thm:equivalence})
{\em Two 
free, finitely generated, strict, $\sP$-graded differential groups 
are $\sO(\sP)$-filtered chain isomorphic
if and only if 
there respective Cartan-Eilenberg systems are isomorphic.
}

\vspace{0.2cm}

An important application of  Theorem A is to the theory of connection matrices developed by R. Franzosa, who introduces
connection matrices in the algebraic topological theory of Morse representations for dynamical systems, cf.\ \cite{fran} and Sect.\ \ref{franmodbr}.
Connection matrices contain information about the connection orbit structure between the Morse sets in a Morse representation and generate the central algebraic structure referred to as a \emph{module braid} in \cite{fran2,fran}, which serves  as an invariant for the Morse representation in question.
In Section \ref{discussion} we prove that Franzosa's module braids in the setting of Morse representations are examples of Cartan-Eilenberg systems. The main result in \cite{fran} states that these module braids are always generated, under some mild conditions, by an appropriately constructed free, graded differential group --- a \emph{connection matrix}.
Theorem A shows that two connection matrices assigned to a Morse representation are conjugate via filtered chain isomorphisms  and thus unique up to filtered conjugacy.
A first partial result in this direction was obtained in \cite{atm} and \cite{hms}. Theorem A  settles the two central conjectures  formulated in \cite{robbin:salamon2}
concerning the similarity of connection matrices, and completes the result in \cite{atm}. 
In Section \ref{discussion} we give an extensive account of these conjectures and their proofs.
Another consequence of Theorem A  concerns  connection matrix theories in the setting of vector spaces. \vspace{0.2cm}

\noindent {\bf Theorem B} (cf.\ Thm.\ \ref{equivKandF})
{\em
    Let $(D,\dff_D)$ be an  $\sO(\sP)$-filtered differential vector space.
    Then, every strict, $\sP$-graded differential vector space $(A,\dff_A)$ 
    whose Cartan-Eilenberg systems is isomorphic to $\bfE(D,\dff_D)$,
        is $\sO(\sP)$-filtered chain equivalent to $(D,\dff_D)$.
    }

\vspace{0.2cm}

As a consequence of Theorem B the connection matrix theories of \cite{fran}, \cite{robbin:salamon2} and \cite{hms} in the setting of finite dimensional differential vector spaces are equivalent. Without the restriction of finite dimensionality the notions in \cite{robbin:salamon2} and \cite{fran} are equivalent.

Another application of Theorem A (Thm.\ \ref{thm:equivalence}) involves the uniqueness of connection matrices for Morse-Smale gradient flows. A $\sP$-graded chain complex $(C,\dff_C)$ of $\Z_2$-vector spaces is Morse-Smale graded if there exists an order-preserving  map $\mu\colon \sP\to \Z$ such that
$\mu(p)=\mu(q)$, $p\neq q$, then $p~\Vert~q$, and $G_pC_k=\Z_2$.
The pair $(\sP,\mu)$ is called a Morse-Smale grading.
For such chain complexes we have the following uniqueness result.
\vspace{0.2cm}

\noindent {\bf Theorem C} (cf.\ Thm.\ \ref{mainuniquethm})
{\em
    Let $(C,\dff_C)$ and $(C,\dff'_C)$ be   algebraic Morse-Smale graded chain complexes with Morse-Smale grading $(\sP,\mu)$ for both chain complexes.
    Then,  $\dff_C=\dff'_C$ if and only if $\bfE(C,\dff_C)\cong \bfE(C,\dff'_C)$.
    }
\vspace{0.2cm}    

Theorem C applies to Morse-Smale gradient flows and provides an algebraic proof of a result by J. Reineck, cf.\ \cite{Reineck}, that connection matrices for such systems are unique, cf.\ Sect.\ \ref{MSunique}.

\begin{remark}
    \label{aboutmodules}
    In this paper we choose to present the results in the category $\sAb$ of abelian groups. The  results remain valid if we consider $R$-modules where the ring $R$ is a principal ideal domain. 
    In partucular, the theory also applies to vector spaces when $R$ is taken to be a field $\F$.
\end{remark}

%%%%%%%%%%%%%%%%%%%%%%%%
%%%%%%%%%%%%%%%%%%%%%%%%%

\section{Graded and filtered differential groups}
\label{diffgroup}
An \emph{ordered decomposition} of an abelian group $C$ 
is a finite poset $\mPi=(\mPi,\le)$ consisting of non-trivial projections $\pi\colon C\to C$ such that
\begin{enumerate}
    \item [(i)] $\pi \circ\pi'=0$ for all distinct pairs $\pi,\pi'\in \mPi$;
    \item [(ii)] $\sum_{\pi\in \mPi} \pi =\id$.
\end{enumerate}
Let $(\sP,\le)$ be a finite poset and let $\sO(\sP)$ be the bounded, distributive lattice of down-sets in  $\sP$.
A \emph{$\sP$-grading} on $C$ 
is an order-embedding 
\[
\grade\colon \mPi\hookrightarrow \sP,
\]
and 
the pair $(C,\grade)$ is called a \emph{$\sP$-graded abelian group}.
A (finite) filtering on an abelian group $C$ is a homomorphism
\[
\flt\colon \sO(\sP) \to \Sub~C,
\]
where $\Sub~C$ is the lattice of subgroups of $C$.
The image 
of $\flt$ is a sublattice consisting of subgroups $F_\alpha C\subset C$ ordered by inclusion.
Note that $\alpha\mapsto F_\alpha C$ being a lattice homomorphism implies that $F_{\alpha\cup\beta} C = F_\alpha C + F_\beta C$ and $F_{\alpha\cap\beta} C = F_\alpha C \cap F_\beta C$.
The pair $(C,\flt)$ is called an \emph{$\sO(\sP)$-filtered group}.
A graded abelian group yields a filtering $\flt\colon \sO(\sP) \to \Sub~C$ of $C$
given by $\alpha \mapsto F_\alpha C:=\bigoplus_{\pi\in \grade^{-1}(\alpha)} \pi C$.
We use the convention that the empty sum is $0$. 
The homomorphism $\flt$ extends to a meet-semilattice homomorphism 
$\beta\smin\alpha\mapsto G_{\beta\smin\alpha} C:=\bigoplus_{\pi\in \grade^{-1}(\beta\smin\alpha)} \pi C$. In particular, this assigns the `atoms' $\{p\}\mapsto G_pC$ and associated decomposition:
$
G_{\beta\smin\alpha}C = \bigoplus_{p\in \beta\smin\alpha} G_pC,
$
with $G_pC=\pi C$, if $\grade^{-1}(\{p\}) =\{\pi\}$ and 
$G_pC=0$ if $\grade^{-1}(\{p\}) =\varnothing$.
We refer to the subgroups $G_p C$ as the subgroups of {homogeneous elements of degree} $p$.
Since $\grade$ need not be onto, some of the subgroups $G_pC$ may be the trivial subgroup $0$ for some $p$ and therefore the subgroups $G_pC$    coincide with the components $\img\pi$ whenever $G_pC$ is non-trivial. It is customary to denote the $\sP$-graded abelian group by its decomposition $C=\bigoplus_{p\in \sP} G_p C$.
A $\sP$-graded abelian group $(C,\grade)$, with $C=\bigoplus_{p\in \sP} G_pC$, 
is \emph{free abelian} if and only if the components $G_pC$ are free abelian.
Conversely, a filtered abelian group also yields a $\sP$-graded abelian group $A=\bigoplus_{p\in \sP} G_pA$, by defining $G_pA \cong F_\beta C/F_\alpha C$, $\beta\smin \alpha=\{p\}$. In general $A$ is not isomorphic to $C$. 
If the groups $G_pA$ are free there exist subspaces $G_pC\subset C$,  with $C_pC\cong G_pA$, such that $F_\alpha C=\bigoplus_{p\in \alpha}G_pC$, cf.\ \cite{robbin:salamon2}.

A homomorphism $f\colon C\to  C'$ of $\sP$-graded abelian groups that satisfies the property $f\bigl(G_pC\bigr) \subset G_q C'$ is called \emph{homogeneous}, or \emph{$\sP$-graded}. A homomorphism $f\colon C\to C'$ of $\sP$-graded, or $\sO(\sP)$-filtered groups  that satisfies $f\bigl(F_\alpha C\bigr)\subset F_\alpha C'$ is called \emph{$\sO(\sP)$-filtered}. 
Note that homomorphisms of graded groups may be filtered when the graded groups are regarded as filtered groups,
with $F_\alpha C=\bigoplus_{p\in \alpha} G_p C$ being the filtering induced by the $\sP$-grading on $C$. 
A homomorphism $f\colon C\to C'$ is an $\sO(\sP)$-filtered isomorphism if
\begin{equation}
    \label{filthomcond}
    f(F_\alpha C) = F_\alpha C',\quad \forall \alpha\in \sO(\sP).
\end{equation}
If the order structure of the grading, or filtering is clear from context we refer to (homogeneous) graded and filtered homomorphisms.
\begin{remark}
    \label{inclusionprop}
    An isomorphism that preserves the filterings 
is not necessarily an filtered isomorphism in the sense of Condition
 \eqref{filthomcond} as the following example shows. Let $C=\Z$ and $C'=\Z$ and $f\colon C\to C'$ is the identity homomorphism.
Suppose $C$ is filtered as $0\subset 4\Z \subset \Z$ and $C'$ is filtered as
$0\subset 2\Z\subset \Z$. Clearly $f$ is an   isomorphism
but $f(4\Z) \subsetneq 2\Z$.
If both  maps $f$ and $f^{-1}$ are $\sO(\sP)$-filtered, then
$F_\alpha C = f^{-1} f(F_\alpha C) \subset f^{-1}(F_\alpha C') \subset F_\alpha C$ which yields \eqref{filthomcond}. Thus filtered isomorphism is an isomorphism for which $f$ and $f^{-1}$ are filtered.
\end{remark}

\subsection{Differential groups}
\label{precdiffgr}
A \emph{differential group} $(C,\dff_C)$ consists of an abelian group $C$  and endomorphism $\dff_C\colon C\to C$  with $\dff_C^2=0$, cf.\ \cite[II.1]{MacLane}. The elements $c\in C$ are referred to as \emph{chains} and the endomorphism $\dff_C$ is referred to as the \emph{differential}. 
We fix the following notation: 
%The set 
$Z(C) := \ker \dff_C$ are  the \emph{cycles} in $C$ and $B(C) := \image \dff_C$ are the \emph{boundaries} in $C$. The \emph{homology} of $(C,\dff_C)$ is defined as $H(C) := Z(C)/B(C)$.
A homomorphism $f\colon (C,\dff_C) \to (A,\dff_A)$ of differential groups is a homomorphism which
satisfies $\dff_A  f = f \dff_C$ and is called a 
\emph{chain homomorphism}, or \emph{chain isomorphism} if $f$ is an isomorphism of abelian groups.
Differential groups and chain homomorphisms form the category $\sAbd$. Homology for a differential group defines a covariant functor $\sH\colon \sAbd \to \sAb$
given by $(C,\dff_C) \mapsto H(C)$, called the \emph{homology functor}. 
The homology may be regarded as a differential group with $\dff=0$.
Next we consider grading and filtering in combination with the differental $\dff_C$.
\begin{definition}
    \label{Pgradedgr}
A  triple $\bigl(C,\dff_C,\grade\bigr)$ is called a {\em $\sP$-graded differential group}
if
\begin{enumerate}
    \item[(i)] $(C,\dff_C)$ is a differential group;
    \item[(ii)] $(C,\grade)$ is a $\sP$-graded  abelian group with grading $\grade\colon \mPi\hookrightarrow \sP$;
    \item[(iii)]  $\dff_C$ is $\sO(\sP)$-filtered.
 \end{enumerate}
A $\sP$-graded differential group is \emph{strict} if 
\begin{enumerate}
    \item [(iv)] $\dff_C F_\alpha C \subset F_{\alpha\smin \{p\}} C$, for all $p$ maximal in $\alpha$.
\end{enumerate}
\end{definition} 
A chain homomorphism, or chain isomorphism between graded or filtered differential groups preserving the filtering is called an \emph{$\sO(\sP)$-filtered chain homomorphism}, or \emph{$\sO(\sP)$-filtered chain isomorphism}.
 The condition on $\dff_C$ to be $\sO(\sP)$-filtered implies that the restriction
    $\dff_C\colon G_qC\to G_pC$ is well-defined. The latter is non-trivial only if $p\le q$. If the differential is strict then the latter inequality is strict, i.e. $p<q$. For any convex set $\beta\smin\alpha$ the pair $(G_{\beta\smin\alpha}C,\dff_C)$ is a differential subgroup.
More generally a triple $(C,\dff_C,\flt)$ is called an \emph{$\sO(\sP)$-filtered differential group} if
$(C,\flt)$ is  $\sO(\sP)$-filtered and the differential satisfies $\dff_C F_\alpha C \subset F_\alpha C$ for all $\alpha\in \sO(\sP)$.
The $\sO(\sP)$-filtering induced by a $\sP$-graded group automatically
satisfies the conditions of a $\sO(\sP)$-filtered differential group
via $F_\alpha C := \bigoplus_{p\in \alpha} G_pC$.
As before the converse is not true in general. 

\begin{example}
\label{chaincomplexex1}
A chain complex $(C,\dff_C)$ is an example of a graded differential group, traditionally denoted by $C=\bigoplus_k C_k$, where the differential $\dff_C$ is homogeneous of degree $-1$, and thus filtered. This is equivalent to
$\dff_C=\bigoplus_k \dff_k$, with  $\dff_k\colon C_k\to C_{k- 1}$ and 
$\dff_k\dff_{k+ 1} =0$.
The chain maps are taken be to  homogeneous, i.e. $f\colon C\to C'$, $f=\bigoplus_k f_k$, with $f_k\colon C_k\to C'_k$ and $\dff'_k f_k = f_{k-1}\dff_k$, for all $k$.
The same holds for cochain complexes.
\end{example}

\begin{remark}
    In the notation for graded and filtered differential groups we will omit $\grade$ and $\flt$ from the notation if there is no ambiguity about the grading or filtering, and simply write $(C,\dff_C)$.
\end{remark}
 \subsection{Cartan-Eilenberg systems} %for  graded differential groups}
\label{CEfordffgr}
In Section \ref{intronew} a basic definition of  Cartan-Eilenberg systems  in the category of abelian groups is given.
Filtered differential groups are a prime source for  generating
excisive Cartan-Eilenberg systems. 
Let $(C,\dff_C)$ be a filtered differential group.
Since the differential $\dff_C$ preserves the filtering we have 
 the short exact sequence:
\begin{equation}
    \label{firstexactseq13}
\begin{tikzcd}
0 \arrow[r] & \displaystyle{\frac{F_\beta C}{F_\alpha C}} \arrow[r, "i"] & \displaystyle{\frac{F_\gamma C}{F_\alpha C}} \arrow[r, "j"] & \displaystyle{\frac{F_\gamma C}{F_\beta C}} \arrow[r] & 0
\end{tikzcd},\quad \alpha\subset \beta\subset\gamma.
\end{equation}
The associated homology is given by
  $E^\beta_\alpha := H(F_\beta C/F_\alpha C)$, $E^\gamma_\alpha := H(F_\gamma C/F_\beta C)$ and
 $E^\gamma_\beta := H\bigl(F_\gamma C/F_\beta C\bigr)$. 
 This yields the exact triangles in \eqref{unpackeddiag12} where 
 the connecting homomorphism  $k\colon H\bigl(F_\beta C/F_\alpha C\bigr) \to H(F_\alpha C)$ 
 is constructed in the usual way. 
 All other axioms are readily verified 
 which yields a Cartan-Eilenberg system denoted
 $\bfE(C,\dff_C)$. %, cf.\ \cite{SpV2}. 
 The excisive property follows from the fact that 
$F_{\alpha\cup\beta}C/F_\alpha C =(F_\alpha C+F_\beta C)/F_\alpha C
\cong F_\beta C/(F_\alpha C \cap F_\beta C) = F_\beta C/F_{\alpha\cap \beta}C$, cf.\ Lem.\ \ref{execiso}. 
The (excisive) Cartan-Eilenberg system $\bfE(C,\dff_C)$ is the Cartan-Eilenberg system of the $\sO(\sP)$-filtered differential group $(C,\dff_C)$.
Since graded differential groups are filtered differential groups
via $F_\alpha C := \bigoplus_{p\in \alpha} G_pC$, we define the Cartan-Eilenberg system of a $\sP$-graded differential group as outlined above. The induced Cartan-Eilenberg system is again denoted by $\bfE(C,\dff_C)$.

\begin{remark}
    \label{homequal}
    For finitely generated, strict $\sP$-graded differential groups the homology groups satisfy $H(G_pC,\dff)= G_pC\cong E^\beta_\alpha$, with $\beta\smin\alpha=\{p\}$ for all $p\in \sP$.
\end{remark}

\section{Induced chain homomorphisms}
\label{inducedch}
%We now turn our attention to proving Theorem \ref{thm:equivalence}. 
In this section we prove two results, Theorems \ref{lem:main} and \ref{lem:main22}, concerning chain homomorphisms that lift homomorphisms on homology.
We start with a number of preliminary lemmas. 

\begin{lemma}
\label{liftinglemma}
Let $D$ be a free abelian group, $(A,\dff_A)$  a differential group,  and let  $h\colon D\to A$ be  a homomorphism such that $\dff_A  h=0$ and $H(h)=0$.\footnote{If $D$ is regarded as  differential group with $\dff=0$, then the condition
%the commutativity of the  diagram 
% \[
% \begin{diagram}
% \dgARROWLENGTH=2.0em
%     \node{D}\arrow{e,l}{0}\arrow{s,l}{h}\node{D}\arrow{s,r}{h}\\
%     \node{A}\arrow{e,l}{\dff_A}\node{A}
% \end{diagram}
% \]
 $\dff_A h=0$ is equivalent to $h$ being a chain homomorphism. The induced homology map is denoted by $H(h)\colon D\to H(A)$.
We say that $D$ is a \emph{boundaryless} differential group.}
Then, $h\colon D\to B(A)$  and there exists a lift $l\colon D\to A$ such that the diagram
\begin{equation}\label{liftdiag12}
    \begin{tikzcd}[column sep=large, row sep=large]
                                          & A \arrow[d, "\dff_A", two heads] \\
D \arrow[r, "h"'] \arrow[ru, "l", dashed] & B(A)                          
\end{tikzcd}
\end{equation}
commutes, i.e. $h= \dff_A l$. 
\end{lemma}

\begin{proof}
By the assumption that $\dff_A h=0$ we have that $h\colon D \to Z(A)$. Observe  the short exact sequence for homology
\begin{equation}
    \label{exactforDD}
    \begin{tikzcd}
0 \arrow[r] & B(A) \arrow[r, "e"] & Z(A) \arrow[r, "\rho"] & H(A) \arrow[r] & 0,
\end{tikzcd}
\end{equation}
where $\rho$ is the projection $x\mapsto [x]\in H(A)$ and $e$ is inclusion.
Consider the diagram:
\begin{equation}
    \label{exactforC}
\begin{tikzcd}[row sep=large]
            & A \arrow[d, "\dff_A"', two heads]  & D \arrow[d, "h"] \arrow[ld, "h"', dashed] \arrow[l, "l"']\arrow[rd,"H(h)"]&             &   \\
0 \arrow[r] & B(A) \arrow[r, "e"]                 & Z(A) \arrow[r, "\rho"]                          & H(A) \arrow[r] & 0
\end{tikzcd}
\end{equation}
From the assumption that  $H(h) = \rho h=0$ %implies that $h(D)\subset B(A)$. From 
and the \emph{exactness} of  \eqref{exactforDD} we conclude that $h$ maps into $B(A)$ and thus $h\colon D \to B(A) \subset Z(A)$.
Since $D$ is a free abelian group it is {projective} and therefore there exists a lift $l\colon D\to A$ such that the diagram in \eqref{liftdiag12} commutes.
\qed
\end{proof}

\begin{remark}
    For the action of differentials we omit parentheses and for other homomorphisms we adopt the usual functional notation, i.e. $\dff c$ and $h(c)$. Parentheses are also omitted in compositions of maps.
\end{remark}
\begin{remark}
In the remainder of the paper $C\oplus C'$ (and $A\oplus A'$) denotes the Cartesian product as direct sum with specifically constructed differentials to emphasize that this is not the product of differential groups with the product differential but a product group with a chosen  differential.
\end{remark}
\begin{remark}
\label{forchaincompl1}
    If $(A,\dff_A)$ is a chain complex and $h_k\colon D\to A_k$ with $\dff_kh_k=0$, then
    $h_k\colon D\to B_k(A)$ and there exists a lift $l_k\colon D\to A_{k+1}$ such that $h_k=\dff_{k+1}l_k$.
\end{remark}
\begin{lemma}
    \label{connhomchar12}
    Let $C,C'$ be abelian groups and consider the
    short exact sequence
    \begin{equation}
        \label{firstshortseq}
    \begin{tikzcd}
0 \arrow[r] & C \arrow[r, "i"] & C\oplus C' \arrow[r, "j"] & C' \arrow[r] & 0,
\end{tikzcd}
\end{equation}
where $i(c) = (c,0)$, $j(c,c') = c'$. Let $\dff_C$ be a differential for $C$ and let $C'$ be boundaryless. Then, for $i$ and $j$ to be chain homomorphisms the differentials on $C\oplus C'$ are of the form
\begin{equation}
    \label{proddiffs1212}
\dff_{C\oplus C'} = \begin{pmatrix} \dff_C & \lambda \\ 0 & 0\end{pmatrix}, 
\end{equation}
where $\lambda\colon C'\to C$ is a chain homomorphisms, i.e. $\dff_C \lambda=0$.
Moreover, for the long exact sequence (exact triangle)
\begin{equation}
    \label{longexactseq1222}
\begin{tikzcd}
{} \arrow[r] & H(C) \arrow[r, "i"] & H(C\oplus C') \arrow[r, "j"] & C' \arrow[r, "H(\lambda)"] & H(C) \arrow[r] & {}
\end{tikzcd}
\end{equation}
the canonical connecting homomorphism is given by $H(\lambda)$.
\end{lemma}
\begin{proof}
    Let us start with the differentials in \eqref{proddiffs1212}. 
In order for $i$ and $j$ to be chain maps we need: $i\dff_C = \dff_{C\oplus C'} i$ and $j \dff_{C\oplus C'} = \dff_{C'} j$. 
By assumption $\dff_{C\oplus C'}$ is given by 
\[
\dff_{C\oplus C'} = \begin{pmatrix} \kappa~ & ~\lambda \\ \mu~ & ~\nu\end{pmatrix},
\]
for some homomorphisms $\kappa\colon C\to C$, $\lambda\colon C'\to C$, $\mu\colon C\to C'$ and $\nu\colon C'\to C'$.
From the conditions on $i$ and $j$, and the fact that $C'$ is boundaryless ($\dff_{C'}=0$), we derive that 
$\bigl(\dff_C c,0\bigr) = i\dff_C c = \dff_{C\oplus C'} ic = \bigl(\kappa c,\mu c \bigr)$, which yields that $\kappa = \dff_C$ and $\mu=0$. Moreover, $\nu c' = j \dff_{C\oplus C'} (c,c') = 0\cdot j(c,c') = 0$, which implies that $\nu = 0$ 
 and therefore $\dff_{C\oplus C'}$ is given by \eqref{proddiffs1212} for some homomorphism $\lambda\colon C'\to C$. 
 Since $\dff_{C\oplus C'}$ is a differential we obtain the condition $\dff_C\lambda =0$, which makes $\lambda$ a chain homomorphism.

As for the long exact sequence we follow the standard construction for the connecting homomorphism. Let $c'\in C'=Z(C')$. By exactness, $c' = j(\tilde c,c')$ for some $\tilde c\in C$. Since $j$ is a chain homomorphism we have that
$j\dff_{C\oplus C'}(\tilde c,c') = 0\cdot j(\tilde c,c')$ and thus
$\dff_{C\oplus C'}(\tilde c,c') \in \ker j = \image i$, which implies that 
$\dff_C \tilde c+\lambda c' =c$
for some $c\in C$. 
Observe that
$\dff_C c 
=\dff_C\bigl( \dff_C \tilde c+\lambda c'\bigr) = \dff_C^2 \tilde c + \dff_C\lambda c' =0$. Define 
the connecting homomorphism $k[c'] := [c] = [\dff_C\tilde c + \lambda c'] = [\lambda c'] = H(\lambda)[c']$, which establishes the connecting homomorphism.
\qed
\end{proof}

\begin{remark}
    \label{forchaincompl2}
    If $(C,\dff_C)$ is a chain complex the short exact sequence in \eqref{firstshortseq} holds degree wise and yields: 
    \begin{equation}\label{forchaincompl3}
    \begin{tikzcd}%[column sep=large, row sep=large]
     & {} \arrow[d]          & {} \arrow[d]  & {} \arrow[d]\\
0 \arrow[r] & C_k \arrow[r,"i_k"] \arrow[d, "\dff_k"] & C_k\oplus C'_k \arrow[r,"j_k"] \arrow[d, "\bar\dff_k"] & C'_k \arrow[r] \arrow[d, "0"] & 0 \\
0 \arrow[r] & C_{k-1} \arrow[r,"i_{k-1}"]   \arrow[d]             & C_{k-1}\oplus C'_{k-1} \arrow[r,"j_{k-1}"]  \arrow[d]              & C'_{k-1} \arrow[r] \arrow[d]               & 0\\
& {}                    & {}      &{}
\end{tikzcd}
\end{equation}
In this case $\bar\dff_k (c,c') = \dff_k c+ \lambda_k c'$, where $\lambda_k\colon C'_k\to C_{k-1}$ satisfies $\dff_{k-1}\lambda_k=0$, and for the long exact sequence $H_k(\lambda_k)$ is the connecting homomorphism. The homology in degree $k$ is given by the functor $H_k$.
\end{remark}

\begin{lemma}\label{lem:kernels}
Consider the following commutative diagram of differential groups with exact rows:
\begin{equation}\label{diag7}
    \begin{tikzcd}[column sep=large, row sep=large]
0 \arrow[r] & C \arrow[r, "i"] \arrow[d, "f'"'] & C\oplus C' \arrow[r, "j"] & C' \arrow[r] \arrow[d, "f''"] & 0 \\
0 \arrow[r] & A \arrow[r, "i"]                 & A \oplus A' \arrow[r, "j"] & A' \arrow[r]                & 0
\end{tikzcd}
\end{equation}
where $i(c) = (c,0)$, $j(c,c') = c'$, 
$f'$ and $f''$ are chain homomorphisms
and where  $C'$ and $A'$ are boundaryless,  $C'$ is a free abelian group and $(C,\dff_C)$ is a differential group.
Then, the differentials on $C\oplus C'$ and $A\oplus A'$ are given by Lemma \ref{connhomchar12} and are of the form
\begin{equation}
    \label{proddiffs}
\dff_{C\oplus C'} = \begin{pmatrix} \dff_C & \lambda \\ 0 & 0\end{pmatrix}, \quad
\dff_{A\oplus A'} = \begin{pmatrix} \dff_A & \lambda' \\ 0 & 0 \end{pmatrix},
\end{equation}
where $\lambda\colon C'\to C$ and $\lambda'\colon A' \to A$ are chain homomorphisms, i.e. $\dff_C \lambda=0$ and $\dff_A\lambda'=0$.
Let $K :=j Z(C\oplus C')\subset C'$ and $K' := j Z(A\oplus A')\subset A'$.
If, 
\begin{equation}
    \label{homcond12}
    H(f')H(\lambda)=H(\lambda') f'',
\end{equation}
then
the map $f''$ induces  a chain homomorphism $f''\colon K\to K'$ by restriction, such that the  diagram 
\begin{equation}\label{diagkernel}
    \begin{tikzcd}[column sep=large, row sep=large]
0 \arrow[r] & Z(C) \arrow[r, "i"] \arrow[d, "f'"'] & Z(C\oplus C') \arrow[r, "j"] & K \arrow[r] \arrow[d, "f''"] & 0 \\
0 \arrow[r] & Z(A) \arrow[r, "i"]                 & Z(A \oplus A') \arrow[r, "j"] & K' \arrow[r]                & 0
\end{tikzcd}
\end{equation}
commutes with exact rows. Moreover, if $C'$ and $A'$ are finitely generated free abelian groups, then there exist complemented subgroups $M\subset C'$, $M'\subset A'$ and isomorphisms $\omega\colon M\to K$ and $\omega'\colon M'\to K'$, such that
the diagram
\begin{equation}\label{diagkernel12}
    \begin{tikzcd}[column sep=large, row sep=large]
0 \arrow[r] & Z(C) \arrow[r, "i"] \arrow[d, "f'"] & Z(C\oplus C') \arrow[r, "\bar j"] 
& M \arrow[r] \arrow[d, "f''"] & 0 \\
0 \arrow[r] & Z(A) \arrow[r, "i"]                & Z(A\oplus A') \arrow[r, "\bar j"]                & M' \arrow[r]                & 0
\end{tikzcd}
\end{equation}
commutes with exact rows, and where $f''\colon M\to M'$ is a chain homomorphism given by restriction.
\end{lemma}
\begin{proof}
By definition $Z(C\oplus C') = \bigl\{ (z,z')\mid \dff_C z+\lambda z'=0\bigr\}$
and the projection $K$ is characterized by 
\begin{equation}
    \label{projZ}
    K=\bigl\{ z'\in C' \mid \dff_C z+\lambda z'=0 \text{~for some~} z\in C\bigr\}.
\end{equation} 
 The same holds for $Z(A\oplus A')$ and $K'$. With the chain maps $i$ and $j$ the rows in \eqref{diagkernel} are exact.
Since $f'$ is a chain homomorphism the restriction of $f'$ to $Z(C)$ defines the homomorphism $f'\colon Z(C) \to Z(A)$. Moreover, the diagram in \eqref{diagkernel} commutes provided that $f''$ is a well-defined homomorphism from $K$ to $K'$.
Consider the compositions $f'\lambda$ and $\lambda'f''$. Then, 
\begin{equation}
    \label{dAhiszero}
     \dff_A f'\lambda = f' d_C\lambda =0,\quad\text{and}\quad
     \dff_A \lambda' f''  = 0,
\end{equation}
which follows from the fact that
 $d_C\lambda = 0$ and $d_A\lambda' = 0$. 
 We conclude that 
$f'\lambda\colon C'\to Z(A)$ and $\lambda'f''\colon C' \to Z(A)$.
Define the homomorphism $h = f'\lambda - \lambda' f''$ acting from $C'$ to $A$. Then, $\dff_A h=0$ and thus
 $h\colon C' \to Z(A)$.
Consider the diagram:
\begin{equation}\label{exactforCa}
    \begin{tikzcd}
            &                  & C' \arrow[d, "h"] \arrow[rd,"H(h)"]&             &   \\
0 \arrow[r] & B(A) \arrow[r, "e"] & Z(A) \arrow[r, "\rho"] & H(A) \arrow[r] & 0
\end{tikzcd}
\end{equation}
Observe that $H(h) =0$. Indeed, since both $f'(\lambda c')\in Z(A)$ and $\lambda'f''(c')\in Z(A)$ the homology class satisfies:
\[
\begin{aligned}
H(h)[c']&=[h(c')] =\rho h(c') = \rho f'(\lambda c') - \rho \lambda' f''(c') =   [f'(\lambda c')] - [\lambda' f''(c')] \\
&=   H(f')[\lambda c'] - H(\lambda') f''(c') = H(f')H(\lambda) [c']-  H(\lambda') f''(c')=0,
\end{aligned}
\]
where $[c']=c'$, and which uses \eqref{homcond12}.

Let $z'\in K$. We show that $f''(z')\in K'$. 
Under restriction of
 $h$ to $K\subset C'$ we have again  $\dff_A h|_K =0$ and $H(h|_K)=0$.
% $[h|_K]=0$.
 Since $C'$ is a free abelian group and $K\subset C'$ a subgroup it follows that
$K$ is free and thus projective.\footnote{Since abelian groups are $\Z$-modules the objects in the category $\sAb$ are free if and only if they are projective. Submodules of a free $\Z$-module are again free and thus projective.}
By Lemma \ref{liftinglemma} 
$h|_K\colon K\to B(A)$ and there exists 
 a lift $l\colon K\to A$ such that $h|_K = \dff_A l$.
For $z'$ define an element $a \in A$, given by $a= l(z')$, so that $\dff_A a = \dff_A l(z') = h|_K(z')$.  
This implies that
$\dff_A a =  f'(\lambda z')- \lambda' f''(z')$ and thus
$f'(\lambda z') = \lambda' f''(z') + \dff_A a$.
For $z'\in K$ the definition of the latter yields a $z\in C$ such that $(z,z')\in Z(C\oplus C')$, i.e. $\dff_C z+\lambda z'=0$.
From $a$ we define
 $\zeta = \bigl(f'(z)+a,f''(z')\bigr)\in A\oplus A'$. Then, by the choice of $a$ and the fact that $f'$ is a chain homomorphism, we have that: 
\[
\begin{aligned}
\dff_{A\oplus A'} \zeta &=\dff_A(f'(z) +a) + \lambda ' f''(z') = f'(\dff_C z) +\big(  \lambda' f''(z') + \dff_A a \big)\\
&=  f'(\dff_C z) + f'(\lambda z') = f'(\dff_C z+\lambda z') = 0,
\end{aligned}
\]
which shows that $\zeta\in Z(A\oplus A')$. By definition $f''(z') = j(\zeta) \in K'$ proving our assertion.
The restriction of  $f''$ to $K$ is obviously a chain homomorphism
since $C'$ and $A'$ are boundaryless.

Since $K\subset C'$ 
is a subgroup of a free abelian group it is also free, with $\rank(K)\le \rank(C')<\infty$, and the basis elements for $K$ can be expressed in terms of a basis for $C'$. To be more precise, using Smith normal form, we can choose a basis
 $\{c'_1,\cdots,c'_r\}$  for $C'$, $r=\rank(C')$, such that
$\{\omega_1 c_1',\cdots, \omega_k c_k'\}$ is a basis for $K$, with $r'=\rank(K)\le \rank(C')=r$. By construction, the subgroup
 subgroup $M\subset C'$, generated by the indicated elements $\{c_1',\cdots,c_{r'}'\}$, is isomorphic to $K$ and an isomorphism $\omega\colon M\to K$ is given as follows: $c'_i \mapsto \omega_i c'_i$, $i=1,\cdots,r'$. By construction $M$ allows a complementary subgroup $M^c = [c'_{r'+1},\cdots,c'_r]$ such that 
$C'=M\oplus M^c$. % by exhausting the remaining basis elements $c_j'$.
The same applies to $K'\subset A'$ 
by choosing a basis $\{a'_1,\cdots,a'_s\}$, $s=\rank(A')$ such that
$K'=[\omega'_1 a'_1,\cdots,\omega'_{s'} a'_{s'}]$, $s'=\rank(K')\le \rank(A')=s$.
As before this yields a decomposition $A'=M'\oplus M'^c$, where $M'\subset A'$ and $M'$ isomorphic to $K'$ with isomorphism $\omega'\colon M'\to K'$ given by $a'_j\mapsto \omega'_j a'_j$, $j=1,\cdots,s'$.
Since $f''$ maps from $K$ to $K'$, for $i=1,\cdots,r'$ we have, $f''(\omega_i c'_i)
=\lambda_{i1}\omega'_1 a'_1 + \cdots + \lambda_{i s'}\omega'_{s'} a'_{s'}$, $f''(c'_i) = \mu_{i1} a'_1 + \cdots + \mu_{is} a'_s$, $i=1,\cdots,r'$, and $f''(\omega_i c'_i)=\omega_i\mu_{i1} a'_1 + \cdots + \omega_i\mu_{is} a'_s$. This implies $f''(\omega_i c'_i)=\omega_i\mu_{i1} a'_1 + \cdots + \omega_i\mu_{i\ell} a'_{s'}$ and thus $\mu_{ij}=0$ for
$j=s'+1,\cdots,s$ and $i=1,\cdots,r'$, and $\omega_i\mu_{ij}=\lambda_{ij}\omega'_j$ for
$j=1,\cdots,s'$ and $i=1,\cdots,r'$. Therefore, $f''\colon M\to M'$. As before the diagram is commutative with exact rows.
 \qed
\end{proof}

\begin{remark}
Equation \eqref{homcond12} yields a commutative diagram in homology
and a cube of  homomorphisms of which the front and back faces are not commutative in general:
\[
\begin{tikzcd}[column sep=large, row sep=large]
C' \arrow[r, "{H(\lambda)}"] \arrow[d, "f''"'] & H(C) \arrow[d, "H(f')"] \\
A' \arrow[r, "{H(\lambda')}"]                 & H(A)               
\end{tikzcd}
\quad\quad
    \begin{tikzcd}[cramped]%[cramped, sep=small]%[column sep=tiny]
                                   & C' \arrow[ld, "0"'] \arrow[rr, "\lambda"] \arrow[d, "f''", no head] &                   & C \arrow[ld, "\dff_C"] \arrow[dd, "f'"] \\
C' \arrow[dd, "f''"'] \arrow[rr, "\lambda" near end] & {} \arrow[d]                                               & C \arrow[dd, "f'" near start]      &                              \\
                                   & A' \arrow[ld, "0"'] \arrow[r, no head, "\lambda'"]                      & {} \arrow[r] & A \arrow[ld, "\dff_A"]            \\
A' \arrow[rr, "\lambda'"]                  &                                                            & A                 &                             
\end{tikzcd}
\]
where $C'=H(C')$, $A'=H(A')$, $f'' = H(f'')$
and  $\lambda$ and $\lambda'$ are chain homomorphisms. 
\end{remark}

\begin{remark}
    \label{forchaincompl4}
    If $(C,\dff_C)$ and $(A,\dff_A)$ are chain complexes we consider the chain homomorphisms
    $f'$ and $f''$ degree wise, i.e.
    $f'_k\colon C_k\to A_k$ and $f''_k\colon C'_k \to A'_k$ for all $k$. In this case
    \eqref{homcond12} becomes
    $H_{k-1}(f'_{k-1})H_k(\lambda_k)=H_k(\lambda'_k) f''_k$,
    with $\lambda_k\colon C'_k\to C_{k-1}$ and $\lambda'_k\colon A'_k\to A_{k-1}$. Diagram \eqref{diagkernel12} holds degree wise with complemented subgroups $M_k$ and $M'_k$ for all $k$.    
\end{remark}

\begin{lemma}\label{lem:gamma}
Consider the commutative diagram:
\begin{equation}\label{dia:lemma8}
    \begin{tikzcd}
                                                 & P' \arrow[ld, "f'"'] \arrow[rr] \arrow[d, "\epsilon'", no head] &                                                 & P'\oplus P'' \arrow[rr] \arrow[d, "\epsilon", no head]      &                                      & P'' \arrow[dd, "\epsilon''"] \arrow[ld, "f''"'] \\
Q' \arrow[dd, "\eta'"] \arrow[rr, phantom] \arrow[rr] & {} \arrow[d]                                          & Q'\oplus Q'' \arrow[rr] \arrow[d, "\eta", no head] \arrow[dd] & {} \arrow[d]                              & Q'' \arrow[d, "\eta''", no head] \arrow[dd] &                                    \\
                                                 & D' \arrow[ld, "g'"] \arrow[r, "i_D", no head]             & {} \arrow[r]                                    & D \arrow[ld, "g"] \arrow[r, "j_D", no head] & {} \arrow[r]                         & D''             \arrow[ld, "g''"]                      \\
B' \arrow[rr, "i_B"']                               &                                                       & B \arrow[rr, "j_B"']                              &                                           & B''                    &                                   
\end{tikzcd}
\end{equation}
where $\eta'$ is surjective, the sequence  $0\xrightarrow[]{~~~} B'\xrightarrow[]{i_B} B\xrightarrow[]{j_B} B''$  is exact and $P''$ is a free abelian group.
The maps on the top part of the diagram are the usual inclusion and projection homomorphisms.
Then, there exists a homomorphism $\gamma \colon P''\to Q'$ such that 
for the choice
\begin{equation}
    \label{constrofF}
f = \begin{pmatrix} f' & \gamma \\ 0 & f'' \end{pmatrix} \colon P'\oplus P'' \to Q'\oplus Q'',
\end{equation}
the  diagram 
\begin{equation}\label{dia:lemma8a}
    \begin{tikzcd}
                                                 & P' \arrow[ld, "f'"'] \arrow[rr] \arrow[d, "\epsilon'", no head] &                                                 & P'\oplus P'' \arrow[rr] \arrow[d, "\epsilon", no head] \arrow[ld, "f"']      &                                      & P'' \arrow[dd, "\epsilon''"] \arrow[ld, "f''"'] \\
Q' \arrow[dd, "\eta'"] \arrow[rr, phantom] \arrow[rr] & {} \arrow[d]                                          & Q'\oplus Q'' \arrow[rr] \arrow[d, "\eta", no head] \arrow[dd] & {} \arrow[d]                              & Q'' \arrow[d, "\eta''", no head] \arrow[dd] &                                    \\
                                                 & D' \arrow[ld, "g'"] \arrow[r, "i_D", no head]             & {} \arrow[r]                                    & D \arrow[ld, "g"] \arrow[r, "j_D", no head] & {} \arrow[r]                         & D''          \arrow[ld, "g''"]                        \\
B' \arrow[rr, "i_B"']                               &                                                       & B \arrow[rr, "j_B"']                              &                                           & B''                    &                                   
\end{tikzcd}
\end{equation}
commutes and $f$ is a lift of $g$ in the sense that $g\epsilon = \eta f$.
\end{lemma}

\begin{proof}
We define a homomorphism
 $\gamma\colon P''\to Q'$ so that $f$ given in \eqref{constrofF}
%$$F = \begin{pmatrix} F' & \gamma \\ 0 & F''\end{pmatrix}$$
is a lift of $g$ in the sense that $ g\epsilon = \eta f$.   Regardless of the definition of $\gamma$, we already have commutativity on $P'\oplus 0$.
Let $(x,0) \in P'\oplus P''$. Then, $f(x,0)=\bigl(f'(x),0 \bigr)$ and a
diagram chase yields
\[
\eta f(x,0) = \eta(f'(x),0) = i_B\eta' f'(x) = i_B g'\epsilon'(x) = g i_D\epsilon' (x) = g\epsilon (x,0) .
\]
It remains to define $\gamma$ on $P''$. Let  $(x,y)\in P'\oplus P''$. In order for $f$ to be a lift of $g$  we need $g\bigl(\epsilon(x,y)\bigr) = \eta f(x,y)$ and therefore
\[
\begin{aligned}
g\epsilon(x,y) &= \eta f(x,y) = \eta \bigl(f'(x)+\gamma y, f''( y)\bigr)\\
&= \eta  \bigl(f'(x), f''(y)\bigr)  + \eta (\gamma y, 0) =  \eta  \bigl(f'(x), f''(y)\bigr)  + i_B \eta ' \gamma y,
\end{aligned}
\]
which implies
\begin{equation}
    \label{condgamma11}
i_B \eta ' \gamma y = g\epsilon(x,y) - \eta  \bigl(f'(x), f''(y)\bigr).
\end{equation}
Using the commutativity with respect to elements $(x,0)\in P'\oplus P''$ one derives that
$g\epsilon (x,0) - 
\eta \bigl(f'(x), 0\bigr)=0$ and thus
\[
\begin{aligned}
 g\epsilon (x,y) &- 
\eta \bigl(f'(x), f''(y)\bigr) \\
&= g\epsilon (x,0) - 
\eta \bigl(f'(x), 0\bigr) +g\epsilon (0,y) - 
\eta \bigl(0, f''(y)\bigr)\\
&= g\epsilon (0,y) - 
\eta \bigl(0, f''(y)\bigr),
\end{aligned}
\]
showing that the expression $g\epsilon (x,y) - 
\eta \bigl(f'(x), f''(y)\bigr)$ is independent of $x$.
Observe that 
\[
\begin{aligned}
j_B\Bigl[ g\epsilon(x,y) - \eta  \bigl(f'(x), f''(y)\bigr) \Bigr] &= 
j_B g \epsilon (x,y) - j_B \eta \bigl(f'(x),f''(y)\bigr)\\
&=
g''j_D \epsilon(x,y) - \eta'' f''( y)\\
&= g''\epsilon''(y) - \eta'' f''( y) = 0,
\end{aligned}
\]
using the fact that $f''$ is a lift of $g''$.
Since $0\xrightarrow[]{~~~} B'\xrightarrow[]{i_B} B\xrightarrow[]{j_B} B''$  is \emph{exact} we have that $\ker j_B = \image i_B$ 
and $i_B$ is injective. Consequently,
\[
\begin{aligned}
b &= g\epsilon (x,y) - 
\eta \bigl(f'(x), f''(y)\bigr)\\
&= g\epsilon (0,y) - 
\eta \bigl(0, f''(y)\bigr) \in \image i_B,
\end{aligned}
\]
independent of $x$. Due to the latter and the injectivity of $i_B$ this defines a homomorphism  $\beta\colon P''\to B'$ given by  $\beta y = i_B^{-1}\bigl( g\epsilon (0,y) - 
\eta \bigl(0, f''(y)\bigr)\bigr)$.
By assumption 
we have the 
\emph{surjective} homomorphism $\eta'\colon Q' \twoheadrightarrow B'$.
Since $P''$ is free abelian, and therefore projective, there
exists a lift $\gamma\colon P'' \to Q'$ such that the diagram
\[
    \begin{tikzcd}[column sep=large, row sep=large]
                                          & Q' \arrow[d, "\eta'", two heads] \\
P'' \arrow[r, "\beta"'] \arrow[ru, "\gamma", dashed] & B'                          
\end{tikzcd}
\]
commutes. Therefore, $\gamma \eta'y = \beta y$ and thus $i_B \gamma \eta' y = i_B\beta y$ which establishes \eqref{condgamma11} and concludes the  construction of $f$.
\qed
\end{proof}

The following theorem concerns $C\oplus C'$ and $A\oplus A'$ not as product of differential groups but with differentials $\dff_{C\oplus C'}$ and $\dff_{A\oplus A'}$ respectively, cf.\ \eqref{proddiffs}, and with embedded differential groups $C$ and $A$ with differential $\dff_C$ and $\dff_A$ respectively, and quotient  groups
$C'$ and $A'$ (boundaryless).

\begin{remark}
    \label{forchaincompl5}
    In the case of chain complexes
    Lemma \ref{lem:gamma} can be applied degree wise which yields homomorphism $f$ in every degree, cf.\ Thm.\ \ref{lem:main}.
\end{remark}

\begin{theorem}\label{lem:main}
Consider the following commutative diagram of {free} differential groups with exact rows:
\begin{equation}\label{diag7-12}
    \begin{tikzcd}[column sep=large, row sep=large]
0 \arrow[r] & C \arrow[r, "i"] \arrow[d, "f'"'] & C\oplus C' \arrow[r, "j"] & C' \arrow[r] \arrow[d, "f''"] & 0 \\
0 \arrow[r] & A \arrow[r, "i"]                 & A \oplus A' \arrow[r, "j"] & A' \arrow[r]                & 0
\end{tikzcd}
\end{equation}
where  $C'$ and $A'$ are  finitely generated and  boundaryless 
differential groups, $i$ and $j$ are given in Lemma \ref{lem:kernels}
and $f'$ and $f''$ are chain homomorphisms.
By Lemma \ref{connhomchar12} the differentials $\dff_{C\oplus C'}$ and $\dff_{A\oplus A'}$ are given by \eqref{proddiffs}.
Assume that the long exact sequences in homology make the commutative diagram
\begin{equation}\label{homdiag}
    \begin{tikzcd}[column sep=large, row sep=large]
{} \arrow[r] & H(C) \arrow[r] \arrow[d, "H(f')"] & H(C\oplus C') \arrow[r] \arrow[d, "g"] & C' \arrow[r, "{H(\lambda)}"] \arrow[d, "f''"] & H(C) \arrow[r] \arrow[d, "H(f')"] & {} \\
{} \arrow[r] & H(A) \arrow[r]                & H(A\oplus A') \arrow[r]                & A' \arrow[r, "{H(\lambda')}"]                & H(A) \arrow[r]                & {}
\end{tikzcd}
\end{equation}
where $g\colon H(C\oplus C') \to H(A\oplus A')$ is a homomorphism, and $H(f')$ and $f''=H(f'')$ are the induced homomorphism on homology and the rows are exact (exact triangles). 
Then, there exists a chain homomorphism $f$ that is a lift for $g$ such that the diagram 
\begin{equation}\label{dia:lem:ses}
    \begin{tikzcd}[column sep=large, row sep=large]
0 \arrow[r] & C \arrow[r] \arrow[d, "f'"] & C\oplus C' \arrow[r] \arrow[d, "f"] & C' \arrow[r] \arrow[d, "f''"] & 0 \\
0 \arrow[r] & A \arrow[r]                & A\oplus A' \arrow[r]                & A' \arrow[r]                & 0
\end{tikzcd}
\end{equation}
is a morphism of short exact sequences of differential groups, i.e. the diagram commutes with short exact rows. 
\end{theorem}

\begin{proof}%[of Theorem \ref{lem:main}]
We construct a homomorphism $f\colon C\oplus C' \to A\oplus A'$ which satisfies two requirements: (i) $f$ is a chain homomorphism which fits into Diagram \eqref{dia:lem:ses}, and (ii) $f$ induces $g$, i.e. $g=H(f)$.
As in the proof of Lemma \ref{lem:kernels} define the subgroups $K=j Z(C\oplus C')$ and $K'=j  Z(A\oplus A')$, which are isomorphic to the complemented subgroups $M$ and $M'$ respectively.
Given the diagram in \eqref{diag7-12} and by
 the commutativity in \eqref{homdiag}, which implies \eqref{homcond12}, Lemma~\ref{lem:kernels} yields the  commutative  diagram in \eqref{diagkernel12}.
Since the rows in   Diagram \eqref{diagkernel12} are exact and both $M$ and $M'$ are free and thus projective\footnote{Both $M\subset C'$ and $M'\subset A'$ and $C'$ and $A'$ are free abelian groups. This implies that both $M$ and $M'$ are free abelian and thus projective.} the short exact sequences split as:
\begin{equation}\label{splitdiag}
    \begin{tikzcd}
0 \arrow[r] & Z(C) \arrow[r] \arrow[d, "f'"] & Z(C)\oplus M \arrow[r] %\arrow[d, "f"]
& M \arrow[r] \arrow[d, "f''"] & 0 \\
0 \arrow[r] & Z(A) \arrow[r]                & Z(A)\oplus M' \arrow[r]                & M' \arrow[r]                & 0
\end{tikzcd}
\end{equation}
i.e. $Z(C\oplus C') \cong Z(C)\oplus M$ and $Z(A\oplus A') \cong Z(A)\oplus M'$. 
To be more precise, since $M$ is projective we have a section map $s\colon M\to Z(C\oplus C')$ and the diagram
\[
\begin{tikzcd}
            &                        &                             & M \arrow[ld,  "s"'] \arrow[d, "\id"] &   \\
0 \arrow[r] & Z(C) \arrow[r, "i", tail] & Z(C\oplus C') \arrow[r, "j", two heads] & M \arrow[r]                       & 0
\end{tikzcd}
\]
such that $j s = \id$. Moreover, $s$ is an isomorphism from $M$ to $s(M)$. By construction $Z(C\oplus C') = 
\ker j \oplus \image s = Z(C) \oplus s(M) \cong Z(C)\oplus M$. The homomorphism $s$ is given by the formula $z' \mapsto (\sigma z',z')$, for some $\sigma\colon M\to C$ such that $\dff_C \sigma z' +\lambda z'=0$, i.e. 
\begin{equation}
    \label{summid12}
    \dff_C\sigma = - \lambda,\quad \text{on~~} M.
\end{equation}
Summarizing we obtain:
\[
\begin{tikzcd}
            &                                         & Z(C) \oplus M \arrow[d, "\cong"',"r"] \arrow[rd, two heads]                     &             &   \\
0 \arrow[r] & Z(C) \arrow[r, "i", tail] \arrow[ru, tail] & Z(C) \oplus s(M) \arrow[r, "j", two heads] \ar[equal]{d} & M \arrow[r] & 0 \\
            &                                         & Z(C\oplus C')                                                          &             &  
\end{tikzcd}
\]
where  both rows are exact  and   the isomorphism $r$ is given by 
\begin{equation}
    \label{isoforK}
r = \begin{pmatrix} \id~ & ~\sigma \\ 0~ & ~\id\end{pmatrix}\colon Z(C)\oplus M \to Z(C\oplus C').
\end{equation}
Indeed, for $c\in Z(C)$ and $z'\in M$, $r(c,z') = (c+\sigma z',z')$ and $\dff_C(c+\sigma z') + \lambda z' = \dff_C \sigma z' + \lambda z'=0$. 
The map $r$ extends to an automorphism on $C\oplus C'$ by defining $\sigma\colon C'\to C$, where we use the fact that $C'=M\oplus M^c$ and define $\sigma|_{M^c}=0$.
The same analysis applies to $Z(A\oplus A')\cong Z(A)\oplus M'$, with  automorphism $r'$ and $\sigma'\colon A'\to A$ with $\sigma'|_{M'^c}=0$.
In order to apply Lemma \ref{lem:gamma} we conjugate the differentials $\dff_{C\oplus C'}$ and $\dff_{A\oplus A'}$:
\[
\bar \dff_{C\oplus C'} := r^{-1} \dff_{C\oplus C'} r,\quad \text{and}\quad \bar \dff_{A\oplus A'} := r'^{-1} \dff_{A\oplus A'} r'.
\]
A direct calculation shows that 
\[
\bar \dff_{C\oplus C'} = \begin{pmatrix} \dff_C & ~~\dff_C\sigma +\lambda\\ 0~ & ~0\end{pmatrix} = \begin{pmatrix} \dff_C & ~~\bar\lambda\\ 0~ & ~0\end{pmatrix} = \begin{pmatrix} \dff_C & ~0\\ 0~ & ~0\end{pmatrix}\quad\text{on $C\oplus M$},
\]
where $\bar\lambda= \dff_C \sigma+\lambda =0$ on $M$. Moreover,
 $\bar \dff_{C\oplus C'} = \dff_{C\oplus C'}$ on $C\oplus M^c$,  i.e. $\bar\lambda|_M=0$ and $\bar\lambda|_{M^c}=\lambda$.
 The same characterization applies to $\bar\dff_{A\oplus A'}$, with $\bar\lambda' = \dff_A\sigma'+\lambda'$.
 This yields $\bar g\colon H(C\oplus C') \to H(A\oplus A')$, where the homology are isomorphic to the homologies in Diagram
 \eqref{homdiag}. 
Now apply Lemma \ref{lem:gamma} with 
Diagrams   \eqref{homdiag} and \eqref{splitdiag} as input. This yields the commutatative diagram:
\begin{equation*}
    \begin{tikzcd}
                                                 & Z(C) \arrow[ld, "f'"'] \arrow[rr] \arrow[d, "~", no head] &                                                 & Z(C)\oplus M \arrow[rr] \arrow[d, "~", no head] \arrow[ld, "\bar f"']      &                                      & M \arrow[dd, "\epsilon''"] \arrow[ld, "f''"'] \\
Z(A) \arrow[dd, "\eta'"] \arrow[rr, phantom] \arrow[rr] & {} \arrow[d]                                          & Z(A)\oplus M' \arrow[rr] \arrow[d, "~", no head] \arrow[dd] & {} \arrow[d]                              & M' \arrow[d, "~", no head] \arrow[dd] &                                    \\
                                                 & H(C) \arrow[ld, "H(f')"] \arrow[r, "~", no head]             & {} \arrow[r]                                    & H(C\oplus C') \arrow[ld, "\bar g"] \arrow[r, "~", no head] & {} \arrow[r]                         & C'              \arrow[ld, "f''"]                     \\
H(A) \arrow[rr, "i"']                               &                                                       & H(A\oplus A') \arrow[rr, "~"']                              &                                           & A'                   &                                   
\end{tikzcd}
\end{equation*}
where the homomorphism 
$\bar f\colon Z(C)\oplus M \to Z(A)\oplus M'$ is given by
\[
\bar f|_{Z(C)\oplus M} =  \begin{pmatrix} f'~ & \bar\gamma|_M \\ 0~ & f''\end{pmatrix}\colon Z(C)\oplus M \to Z(A)\oplus M',
\]
for some homomorphism $\bar\gamma|_M\colon M\to Z(A)$, with $\dff_A \bar\gamma|_M=0$,
is a chain homomorphism, since the differentials on $Z(C)\oplus M$ and $Z(A)\oplus M'$ vanish.
Moreover, $\bar f$ is a lift for the homology map $\bar g\colon H(C\oplus C') \to H(A\oplus A')$.
Because $f'$ and $f''$ are defined on $C$ and $C'$ respectively it remains to extend $\bar \gamma|_M$ to all of $C'$ in order to construct a chain homomorphism $\bar f\colon C\oplus C'\to A\oplus A'$. 
As pointed out above $M$ is free and complemented with
$C'=M\oplus M^c$, and $M^c$ free and thus projective.
We now construct $\bar \gamma$ on $M^c$. % since $\dff_A \gamma|_M$ vanishes on $M$.
As before we write $\bar h = f'\bar\lambda - \bar\lambda' f'' = h + f'\dff_C\sigma - \dff_A\sigma' f''$. Note that $\bar h|_M=0$.
By \eqref{dAhiszero},
$\dff_A \bar h=0$ and $H(\bar h)=0$, cf.\ Lem.\ \ref{lem:kernels}.
This applies  in particular to the restriction $\bar h|_{M^c}$. Since $M^c$ is projective we can apply Lemma \ref{liftinglemma}  (using that the differential on $C'$ and thus $M^c$ is zero) which yields the homomorphism
$\bar h|_{M^c}\colon M^c\to B(A)$ and a lift $\bar \gamma|_{M^c}\colon M^c \to A$ 
\[
    \begin{tikzcd}[column sep=large, row sep=large]
                                          & A \arrow[d, "\dff_A", two heads] \\
M^c \arrow[r, "\bar h|_{M^c}"'] \arrow[ru, "\bar \gamma|_{M^c}", dashed] & B(A)                   %      
\end{tikzcd}
\]
such that
$\dff_A\bar\gamma|_{M^c} =\bar h|_{M^c}$ and thus $\dff_A\bar\gamma =\bar h$. 
The homomorphism $\bar\gamma$ is now defined on all of $C'$. 
This completes the construction of $\bar f\colon C\oplus C' \to A\oplus A'$.
The restriction to $Z(C)\oplus M$ is a lift for $\bar g$ and therefore $\bar f$ is a lift for
$\bar g$.
Verifying that $\bar f$ is a chain homomorphism is equivalent to verifying that $f= r' \bar f r^{-1}$ is a chain homomorphism.
Observe that $f$ is given by
\begin{equation}
    \label{formoff}
f= r' \bar f r^{-1} = \begin{pmatrix} f'~ & -f'\sigma+\sigma' f''+\bar\gamma \\ 0~ & f''\end{pmatrix}= \begin{pmatrix} f'~ & \gamma \\ 0~ & f''\end{pmatrix}\colon C\oplus C' \to A\oplus A'
\end{equation}
 With this choice of $f$ Diagram \eqref{dia:lem:ses} commutes regardless of the choice of $\gamma$.  
In order for $f$ to be a chain homomorphism we need $\dff_{A\oplus A'} f = f \dff_{C\oplus C'}$. The compositions are given by
\begin{align*}
\dff_{A\oplus A'} f = \begin{pmatrix}  \dff_A & \lambda' \\ 0 & 0\end{pmatrix} \begin{pmatrix} f' & \gamma \\ 0 & f''\end{pmatrix} = \begin{pmatrix} \dff_A f' & \dff_A\gamma + \lambda' f''\\ 0 & 0\end{pmatrix}
\\
f \dff_{C\oplus C'} = \begin{pmatrix} f' & \gamma \\ 0 & f''\end{pmatrix} \begin{pmatrix} \dff_C & \lambda \\ 0 & 0\end{pmatrix} = \begin{pmatrix}  f'\dff_C & f'\lambda\\ 0 & 0\end{pmatrix}
\end{align*}
This yields the condition:
\begin{equation}
    \label{condforhom}
\dff_A\gamma = f'\lambda - \lambda' f''=h.
\end{equation}
Now, $\dff_A\gamma = -\dff_A f' \sigma + \dff_A\sigma' f'' + \dff_A\bar\gamma = -f'\dff_C\sigma + \dff_A\sigma' f'' + \bar h$.
On $M$ we obtain $\dff_A\gamma = -f'\dff_C\sigma + \dff_A\sigma' f'' = f'\lambda-\lambda'f''=h$, and on $M^c$, $\dff_A \gamma =  -f'\dff_C\sigma + \dff_A\sigma' f'' + \bar h = h$, which realizes \eqref{condforhom}. By construction the chain homomorphism $f$ is a lift for $g$ which completes the proof.
\qed
\end{proof}

\begin{remark}
    \label{forchaincompl6}
    As for Lemma \ref{lem:gamma}, in the case of chain complexes
   Theorem \ref{lem:main} can be applied degree wise which yields homomorphism $f$ in every degree.
\end{remark}

Consider subgroups $C,C'\subset \overline C$ and $A,A'\subset \overline A$. The subgroups $C,C'$ and $A,A'$ are equipped with differentials $\dff_C, \dff_{C'}$ and  $\dff_A, \dff_{A'}$ respectively.
 The next result concerns the sums $C+C'$ and $A+A'$ as differential groups. 
\begin{theorem}\label{lem:main22}
Let $C,C'\subset \overline C$ and $A,A'\subset \overline A$ be subgroups.
Consider the following  diagram  with exact rows:
\begin{equation}\label{diag7-1244}
    \begin{tikzcd}[column sep=large, row sep=large]
0 \arrow[r] & C\cap C' \arrow[r, "\varrho"] \arrow[d, "\bar f"'] & C\times C' \arrow[d, "f'\times f''"']\arrow[r, "\varpi"] & C+C' \arrow[r] \arrow[d, "f"] & 0 \\
0 \arrow[r] & A\cap A' \arrow[r, "\varrho"]                 & A \times A' \arrow[r, "\varpi"] & A+A' \arrow[r]                & 0
\end{tikzcd}
\end{equation}
where $\varrho(c) = (c,-c)$ and $\varpi(c,c') = c+c'$. We equip
 $C,C'$ and $A,A'$  with differentials $\dff_C, \dff_{C'}$ and  $\dff_A, \dff_{A'}$ respectively. The product differentials are given by
\[
\dff_{C\times C'} = \dff_C\times \dff_{C'} = \begin{pmatrix} \dff_C & 0 \\ 0 & \dff_{C'}\end{pmatrix},\quad\text{and}\quad \dff_{A\times A'} = \dff_A\times \dff_{A'} = \begin{pmatrix} \dff_A & 0 \\ 0 & \dff_{A'}\end{pmatrix}.
\]
The homomorphisms $f'\colon C\to A$ and $f''\colon C'\to A'$ are 
chain homomorphisms with respect to $\dff_C$ and $\dff_A$, and $\dff_{C'}$ and $\dff_{A'}$ respectively. %,  and the product differentials are given by
Then,
\begin{enumerate}
    \item [(i)] the maps $\varrho$ and $\varpi$ are chain homomorphisms if and only if
    $\dff_{C\cap C'}c = \dff_C c=\dff_{C'} c$ for all $c\in C\cap C'$ and $\dff_{C+C'} (c+c') := \dff_C c+\dff_{C'} c'$ for all $c\in C$ and all $c'\in C'$, and the same for $\dff_{A\cap A'}$ and $\dff_{A+A'}$;
    \item [(ii)] Diagram \eqref{diag7-1244} is commutative and the maps $\bar f$ and $f$ are chain homomorphisms if and only if
    $\bar f(c) = f'(c)=f''(c)$ for all $c\in C\cap C'$ and $f(c+c') := f'(c) + f''(c')$ for all $c\in C$ and all $c'\in C'$.
\end{enumerate}
Given the commutative diagrams in homology:
\begin{equation}
  \label{giveninduced}  
\begin{tikzcd}[column sep=large, row sep=large]
H(C) \arrow[r, "i"] \arrow[d, "H(f')"'] & H(C+C') \arrow[d, "g"] \\
H(A) \arrow[r, "i"]                 & H(A+A')               
\end{tikzcd}
\quad\quad
\begin{tikzcd}[column sep=large, row sep=large]
H(C') \arrow[r, "i"] \arrow[d, "H(f'')"'] & H(C+C') \arrow[d, "g"] \\
H(A') \arrow[r, "i"]                 & H(A+A')               
\end{tikzcd}
\end{equation}
where $i([c]_C)= [c]_{C+C'}$ and  $i([c']_{C'})= [c']_{C+C'}$, and analogously for $A,A'$ and $A+A'$. Then, 
\begin{enumerate}
    \item [(iii)] the chain homomorphism $f$ lifts $g$, i.e. $g=H(f)$.
\end{enumerate}
\end{theorem}

\begin{proof}
Let us start with the observation that the sum of $f'$ and $f''$ is well-defined 
under the assumption that $f'=f''$ on $C\cap C'$. The same applies to
the sum of differentials $\dff_C$ and $\dff_{C'}$, and $\dff_A$ and $\dff_{A'}$.
Consider $c+c'=\bar c +\bar c'$, $c,\bar c\in C$ and $c',\bar c'\in C'$. This implies that $c-\bar c=c'-\bar c'\in C\cap C'$. We show that $f'(c)+f''(c')=f'(\bar c)+f''(\bar c')$. Consider,
\[
f'(c)-f'(\bar c) +f''(c')-f''(\bar c') = f'(c-\bar c)-f''(c'-\bar c')=0,
\]
since $c-\bar c=c'-\bar c'\in C\cap C'$ and $f'=f''$ on $C\cap C'$, which proves $f\colon C+C'\to A+A'$ is well-defined. 
Similarly, $\dff_{C+C'}\colon C+C'\to C+C'$ and $\dff_{A+A'}\colon A+A'\to A+A'$
are well-defined.

For the diagrams:
\[
\begin{tikzcd}[column sep=large, row sep=large]
C\cap C' \arrow[r, "\varrho"] \arrow[d, "\dff_{C\cap C'}"'] & C\times C' \arrow[d, "\dff_{C\times C'}"] \\
C\cap C' \arrow[r, "\varrho"]                 & C\times C'              
\end{tikzcd}
\quad\quad\begin{tikzcd}[column sep=large, row sep=large]
C\times C' \arrow[r, "\varpi"] \arrow[d, "\dff_{C\times C'}"'] & C+C' \arrow[d, "\dff_{C+C'}"] \\
C\times C' \arrow[r, "\varpi"]                 & C+C'             
\end{tikzcd}
\]
to commute we obtain $\dff_{C\times C'} (c,-c) = \bigl( \dff_C c,-\dff_{C'} c\bigr)
= \bigl( \dff_{C\cap C'} c, -\dff_{C\cap C'} c\bigr)$ for all $c\in C\cap C'$, which is equivalent to  
$\dff_{C\cap C'}=\dff_C=\dff_{C'}$. By the same token the second diagram yields
$\dff_{C+C'}(c+c') = \dff_C c+\dff_{C'} c'$ which is well-defined by our above observation. The same arguments apply to $\dff_A,\dff_{A'}$, $\dff_{A\cap A'}$ and  $\dff_{A+A'}$
 concluding the proof of (i).

Consider the diagrams:
\[
\begin{tikzcd}[column sep=large, row sep=large]
C\cap C' \arrow[r, "\varrho"] \arrow[d, "\bar f"'] & C\times C' \arrow[d, "f'\times f''"] \\
A\cap A' \arrow[r, "\varrho"]                 & A\times A'              
\end{tikzcd}
\quad\quad\begin{tikzcd}[column sep=large, row sep=large]
C\times C' \arrow[r, "\varpi"] \arrow[d, "f'\times f''"'] & C+C' \arrow[d, "f"] \\
A\times A' \arrow[r, "\varpi"]                 & A+A'             
\end{tikzcd}
\]
In order for the first diagram to commute we chase the diagram which yields the condition
$(f'\times f'') (c,-c) = \bigl( f'(c),-f''(c')\bigr)
= \bigl( \bar f(c),-\bar f(c)\bigr)$ for all $c\in C\cap C'$. This condition
 is equivalent to $\bar f(c) = f'(c)=f''(c)$ for all $c\in C\cap C'$. 
As for the second diagram we obtain $f(c+c') = f'(c)+f''(c')$ for all $c\in C$ and $c'\in C'$, and $f$ is well-defined by the above observation.
 It remains to show that $\bar f$ and $f$ satisfy the chain homomorphism property. Since $f'$ and $f''$ are chain homomorphisms, and the differential coincide on the intersections, $\bar f$ is a chain homomorphism. For $f$ we have
\[
\begin{aligned}
f\bigl( \dff_{C+C'}(c+c')\bigr) &=f\bigl( \dff_{C} c+\dff_{C'} c'\bigr)
= f'(\dff_C c) + f''(\dff_{C'} c')\\ 
&= \dff_{A} f'(c) + \dff_{A'} f''(c')
= \dff_{A+A'} \bigl(f'(c)+f''(c')\bigr)\\
&= \dff_{A+A'} f(c+c'),
\end{aligned}
\]
 proving (ii).

As for (iii) we argue as follows.
From \eqref{giveninduced} we have that $g\bigl([c]_{C+C'}\bigr) = [f'(c)]_{A+A'}$, $c\in C$ and $g\bigl([c']_{C+C'}\bigr) = [f''(c')]_{A+A'}$, $c'\in C'$. 
Consequently, 
\[
\begin{aligned}
g\bigl([c+c']_{C+C'}\bigr) &= g\bigl([c]_{C+C'}\bigr)+ g\bigl([c']_{C+C'}\bigr) = [f'(c)]_{A+A'}+[f''(c')]_{A+A'}\\
&= [f'(c)+f''(c')]_{A+A'}
= [f(c+c')]_{A+A'}\\
&= H(f)\bigl([c+c']_{C+C'}\bigr),
\end{aligned}
\]
which completes the proof.
    \qed
\end{proof}

\begin{remark}
    \label{forchaincompl7}
    As before Theorem \ref{lem:main22} can be carried out degree wise which is of particular interest for chain complexes.
\end{remark}

\begin{remark}
    The Lemma \ref{liftinglemma}, \ref{connhomchar12} and \ref{lem:gamma} hold for $R$-modules over any commutative ring $R$ and replacing free abelian with projective. Lemma \ref{lem:kernels} requires submodules of projective modules to be projective which dictates $R$ to be an hereditary ring.
    Theorem \ref{lem:main} can therefore also be stated for $R$-modules over an hereditary ring $R$ with free abelian replace by projective.
    Theorem \ref{lem:main22} holds for $R$-modules over any commutative ring $R$.
    Consequently, all of the above mentioned lemmas and theorems hold if we consider rings $R$ that are principal ideal domains. 
    An example of a principal ideal domain is a field $\F$ which implies that all of the above results remain valid for $\F$-vector spaces.
    If for example  $R=\Z$ we retrieve the category of abelian groups. More general principal ideals are for instance polynomial rings $\F[x]$ over a field $\F$.
\end{remark}

\begin{remark}
    \label{chaincomplexadd}
    Chain complexes are the primary sources for differential groups, especially in their applications to dynamical systems. As noted in Example \ref{chaincomplexex1}, a chain complex is a graded differential group with an off-diagonal differential. This characteristic ensures that the homology is also graded: $H(C)=\bigoplus_k H_k(C)$, where $H_k(C)=
    \ker \dff_k/\image \dff_{k+1}$. 
The chain maps for chain complexes are homogeneous and induce homogeneous homomorphisms on homology: $H(f) = \bigoplus H_k(f_k)$ with $H_k(f_k)\colon H_k(C)\to H_{k-1}(C)$. 
An $\sO(\sP)$-filtering is defined by $\alpha\mapsto F_\alpha C_k$ for all $k$ --- degree wise filtering --- which is a filtering on $C$, but not vice versa. Chain homomorphisms are assumed to be filtered degree wise: $f_k(F_\alpha C_k) \subset F_\alpha C'_k$ for all $k$.
The morphisms for the associated Cartan-Eilenberg systems are homogeneous with respect to the integer grading induced by the chain complexes.
The assertions of Theorems \ref{lem:main} and \ref{lem:main22} remain valid if the differential groups are chain complexes.
 The techniques used in the proofs of Theorems \ref{lem:main} and \ref{lem:main22} allow us to construct the chain maps separately in each degree. 
 In the previous paragraphs we indicated the restriction to chain complexes in Remarks \ref{forchaincompl1}, \ref{forchaincompl2}, \ref{forchaincompl4}, \ref{forchaincompl5}, \ref{forchaincompl6} and \ref{forchaincompl7}.
 For simplicity of exposition, we do not use the specific grading of chain complexes.
\end{remark}

\section{Properties of excisive Cartan-Eilenberg systems}
\label{PropsCES}
In Section \ref{intronew} we introduced the notion of Cartan-Eilenberg system over a general poset. In this section we will provide a formal definition and derive   some essential properties needed to prove the main theorem. 

\subsection{A categorical definition}
\label{catdefn}
Let $(\sQ,\le)$ be a poset, or more generally a pre-order. We may regard $\sQ$ as small (thin) category, denoted $\IIonei$, where the objects in the category are the elements in $\sQ$ and the order relations $\alpha\le \beta$ account for the morphisms, or arrows, i.e. $\alpha\le \beta$ yields the arrow $\alpha \to \beta$. 
%We denote this category by $\IIFi$. 
The \emph{arrow category} of $\sQ$ consists of pairs $(\alpha,\beta)$, with $\alpha\le \beta$, 
and unique morphisms $(\alpha,\beta) \to (\gamma,\delta)$ for $\alpha\le \gamma$ and 
$\beta\le \delta$,
and is denoted by $\IIi$. The latter corresponds to commutative diagrams in $\sQ$. The (functor) categories $\IIni$ consist of concatenations of $n-1$ arrows and are given by $n$-tuples\footnote{The entries $\alpha_i$ in $(\alpha_1,\cdots,\alpha_{n})$ satisfy $\alpha_1\le\cdots\le\alpha_{n}$ and are not necessarily distinct elements of $\sQ$.} $(\alpha_1,\cdots,\alpha_{n}) \in \sQ\times\cdots\times \sQ$, $\alpha_i\in \sQ$, with the additional requirement that
$\alpha_1\le\cdots\le\alpha_{n}$.\footnote{The notation $\IIni$ incates the functor category of functors $F\colon {\mathbf n} \to \sQ$, where ${\mathbf n}$ is the category $\bullet\to\bullet\to\cdots\to\bullet\to\bullet$ consisting of $n+1$ objects and $n$ arrows. The objects are identified with nested $n$-tuples in $\sQ$ and the morphisms are given by the induced order on $n$-tuples.}
An object in $\IIni$ may also be regarded as a path of length $n$ in the directed graph induced by the poset $\sQ$. 
%A tuple is an infinite or bi-infinite path. 
Following \cite{HelleRognes} we consider the covariant functors $\spi_0$, $\spi_1$ and $\spi_2$ acting from $\IIthreei$ to $\IIi$,
given by 
% $(\alpha,\beta,\gamma) \xmapsto{\spi_0} (\alpha,\beta)$, $(\alpha,\beta,\gamma) \xmapsto{\spi_1} (\alpha,\gamma)$ and
% $(\alpha,\beta,\gamma) \xmapsto{\spi_2} (\beta,\gamma)$
\[
(\alpha,\beta,\gamma) \xmapsto{\spi_0} (\alpha,\beta),\quad (\alpha,\beta,\gamma) \xmapsto{\spi_1} (\alpha,\gamma),\text{~~and~~} (\alpha,\beta,\gamma) \xmapsto{\spi_2} (\beta,\gamma),
\]
respectively, and  the natural transformations 
$\imath\colon \spi_0\Rightarrow\spi_1$ and
$\jmath\colon \spi_1\Rightarrow\spi_2$ 
whose components 
are given by
\[
(\alpha,\beta) \xmapsto{\imath} (\alpha,\gamma) \text{~~and~~}
(\alpha,\gamma) \xmapsto{\jmath} (\beta,\gamma),
\]
respectively.
\begin{definition}
    \label{CEsys}
A \emph{Cartan-Eilenberg system}\footnote{Cartain-Eilenberg systems can be formulated in any abelian category such as $R$-modules or $\F$-vector spaces.} over a poset $\sQ$ consists of a covariant functor $\sE\colon \IIi\to \sRmod$   and a natural transformation $\sk\colon \sE \spi_2 \Rightarrow \sE\spi_0$ between the composite functors $\sE\spi_2$ and $\sE\spi_0$,
called the \emph{connecting homomorphism},  
%called the \emph{differential},  
such that 
\[
\begin{tikzcd}[column sep=small]
\sE \spi_0 \arrow[rr, "\sE\imath", Rightarrow] &                               & \sE\spi_1\arrow[ld, "\sE\jmath", Rightarrow] \\
                              & \sE\spi_2 \arrow[lu, "\sk", Rightarrow] &                              
\end{tikzcd}
\]
is an exact triangle, where the natural transformations $\sE\imath$ and $\sE\jmath$ are the right whiskerings of $\sE$ and $\imath$, and $\sE$ and $\jmath$ respectively.
A Cartan-Eilenberg system over $\sQ$ is denoted by $\bfE = \bigl(\IIi,\sE,\sk\bigr)$.
\end{definition}

\begin{remark}
    In \cite{CE} the poset $\sQ$ are the extended integers. The above definition is based on the extension in \cite{HelleRognes} for linearly ordered posets. 
\end{remark}

Morphism between Cartan-Eilenberg  systems  $\bfE$ and $\bfE'$ are given as natural transformations $\sh\colon \sE \Rightarrow \sE'$ between the functors $\sE$ and $\sE'$ such that %the right diagram
\begin{equation}
    \label{nattransform12}
\begin{tikzcd}
\sE\spi_2 \arrow[r, "\sk", Rightarrow] \arrow[d, "\sh\spi_2"', Rightarrow] & \sE\spi_0 \arrow[d, "\sh\spi_0", Rightarrow] \\
\sE'\spi_2 \arrow[r, "\sk'", Rightarrow]                             & \sE'\spi_0                           
\end{tikzcd}
\end{equation}
is commutative, where $\sh\spi_2$ and $\sh\spi_1$ are the left whiskerings of $\sh$ and $\spi_2$ and $\spi_0$ respectively.
The components are given in \eqref{dia:exacttriangles}.

\subsection{Incomparable convex sets}
\label{relexsys}
We restrict here to excisive Cartan-Eilenberg systems over a finite distributive lattice $\sQ=\sO(\sP)$, where $\sP$ is a finite poset.
Recall that a Cartan-Eilenberg system is excisive if $\ell\colon E_{\alpha\cap\beta}^\alpha \xrightarrow[]{\cong} E_\beta^{\alpha\cup\beta}$ is an isomorphism for all $\alpha,\beta\in \sO(\sP)$.
\begin{lemma}
    \label{execiso}
    Let $(\alpha,\beta),(\alpha',\beta')\in \IIi$ with $\beta\smin\alpha=\beta'\smin\alpha'$.   Then,
    $E^\beta_\alpha \cong E^{\beta'}_{\alpha'}$.
\end{lemma}
\begin{proof}
    Define $\tilde\alpha=\alpha\vee\alpha'$ and $\tilde\beta=\beta\vee\beta'$. Then,
$\alpha,\alpha'\le\tilde\alpha$, $\beta,\beta'\le \tilde\beta$ and 
$\beta\smin\alpha=\beta'\smin\alpha'=\tilde\beta\smin\tilde\alpha$. 
Observe that $(\beta\cap\tilde\beta)\smin (\alpha\cup\tilde\alpha)=\beta\smin \tilde\alpha = 
\tilde\beta\smin\tilde\alpha$. Consequently,  $(\beta\cup\tilde\alpha)\smin\tilde\alpha = \tilde\beta\smin\tilde\alpha$ and therefore $\tilde\beta = \beta\cup\tilde\alpha$. Similarly, $\beta\smin(\beta\cap\tilde\alpha) = \beta\smin\alpha$ which implies that $\alpha = \beta\cap \tilde\alpha$.
By the excisive property we obtain
$
E^\beta_\alpha = E^\beta_{\tilde\alpha\cap \beta} \xrightarrow[]{\ell}  E^{\tilde\alpha\cup\beta}_{\tilde\alpha}
= E^{\tilde\beta}_{\tilde\alpha},
$
By the same token one proves that 
$E^{\beta'}_{\alpha'} = E^{\beta'}_{\tilde\alpha\cap \beta'} \xrightarrow[]{\ell'}  E^{\tilde\alpha\cup\beta'}_{\tilde\alpha}
= E^{\tilde\beta}_{\tilde\alpha},
$ and thus $(\ell')^{-1}\ell\colon E^\beta_\alpha \xrightarrow[]{\cong} E^{\beta'}_{\alpha'}$.
\qed
\end{proof}
In an excisive Cartan-Eilenberg  system  the isomorphisms between $E$-terms yield additional relations on the maps $i$ and $j$.
\begin{lemma}
    \label{idforincomp}
    Consider non-trivial triples
 $(\alpha,\beta,\gamma), (\alpha,\beta',\gamma)$, i.e. $\alpha\subsetneq \gamma$,  such that $\beta\smin\alpha = \gamma\smin\beta'$  and $\beta'\smin\alpha = \gamma\smin\beta$. 
    Then, in the exact triangles
 \begin{equation}
    \label{exact3abc}
    \begin{tikzcd}[column sep=0.5cm]
E^{\beta}_{\alpha} \arrow[rr, "i"] && E^{\gamma}_{\alpha} \arrow[ld, "j"'] \\
& E^{\gamma}_{\beta} \arrow[lu, "k"']
\end{tikzcd}
\qquad
\begin{tikzcd}[column sep=0.5cm]
E^{\beta'}_{\alpha} \arrow[rr, "i'"] && E^{\gamma}_{\alpha} \arrow[ld, "j'"'] \\
& E^{\gamma}_{\beta'} \arrow[lu, "k'"']
\end{tikzcd}
\end{equation}
the homomorphisms $i,i'$ and $j,j'$
satisfy: $\ell^{-1}j'i=\id$ on $E_\alpha^\beta$ and $ji'\ell'^{-1}=\id$ on $E_\beta^\gamma$, where
$\ell\colon E^{\beta}_{\alpha}\to E^{\gamma}_{\beta'}$ and $\ell'\colon E^{\beta'}_{\alpha} \to E^{\gamma}_{\beta}$ are the isomorphisms by the excisive property.
\end{lemma}
\begin{proof}
An excisive Cartan-Eilenberg system yields the isomorphisms $\ell\colon E^{\beta}_{\alpha}\to E^{\gamma}_{\beta'}$ and $\ell'\colon E^{\gamma}_{\beta} \to E^{\beta'}_{\alpha}$.
By the definition of the functor $\sE$ and the categorical interpretation ordered pairs, $j'i$ is the 
homomorphism $\ell\colon E^{\beta}_{\alpha}\to E^{\gamma}_{\beta'}$, which
 implies that $j'  i$ is an isomorphism and 
 $\ell^{-1}j'i = \id$ on $E^\beta_\alpha$.
The same applies to isomorphism $\ell'\colon  E^{\beta'}_{\alpha}\to E^{\gamma}_{\beta}$
and $j  i' {\ell'}^{-1} = \id$ on $E^{\gamma}_\beta$.
 \qed
\end{proof}
\begin{remark}
 \label{convexpara}   
 For non-trivial triples $(\alpha,\beta,\gamma), (\alpha,\beta',\gamma)$, such that $\beta\smin\alpha = \gamma\smin\beta'$  and $\beta'\smin\alpha = \gamma\smin\beta$,
the convex sets $\xi=\beta\smin \alpha$ and $\eta=\gamma\smin\beta$ are \emph{incomparable}, or \emph{parallel}, i.e. $\xi$ and $\eta$ are incomparable with respect induced order by $\sP$.
They satisfy the additional property that
$\xi\cap\eta=\varnothing$ and $\xi\cup\eta=\zeta=\gamma\smin\alpha$, as well as $\beta\cap \beta'=\alpha$ and $\beta\cup \beta' = \gamma$. In pariticular, $\beta\neq \beta'$.
\end{remark}

    For the $E$-terms of the
     triples 
$(\alpha,\beta,\gamma), (\alpha,\beta',\gamma)$ in Lemma \ref{idforincomp} 
the homomorphisms $i,i'$ are injective and the homomorphisms $j,j'$ are surjective. In particular, using Lemma \ref{idforincomp},
\[
\begin{tikzcd}
0 \arrow[r] & E_\alpha^\beta \arrow[r, "i", tail] & E^\gamma_\alpha \arrow[r, "j", two heads] & E^\gamma_\beta \arrow[r] \arrow[l, "i'\ell'^{-1}", tail, bend left, shift left=1] & 0,
\end{tikzcd}
\]
which shows that the exact sequence is right split and thus a split exact sequence with $E^\gamma_\alpha \cong E^\beta_\alpha \oplus E^\gamma_\beta$.
The same construction follows by using the other triple.
The connecting homomorphisms $k$ and $k'$ are trivial proving that the exact triangle in \eqref{exact3abc} are in fact split exact sequences.
These consideration hold for excisive Cartan-Eilenberg systems in general.
If we consider Cartan-Eilenberg systems generated by a $\sP$-graded differential group we can derive properties of the differential for the above case of triples describing incomparable convex sets.
As before consider triples
 $(\alpha,\beta,\gamma), (\alpha,\beta',\gamma)$, such that $\beta\smin\alpha = \gamma\smin\beta'$  and $\beta'\smin\alpha = \gamma\smin\beta$.
Let $C=\bigoplus_{p\in \sP} G_p C$, then \eqref{firstexactseq13} yields the exact sequences
\begin{equation}
    \label{twoexactseq}
\begin{tikzcd}
            & G_{\gamma\smin\beta'} C \ar[equal]{d}                 & G_{\beta\smin\alpha} C \oplus G_{\gamma\smin\beta} C \ar[equal]{d}                 &             &   \\
0 \arrow[r] & G_{\beta\smin \alpha}C \arrow[r, "i"]                                 & G_{\gamma\smin\alpha} C \arrow[r, "j"] \ar[equal]{d}  & G_{\gamma\smin\beta} C \arrow[r] & 0 \\
0 \arrow[r] & G_{\beta'\smin\alpha} C \arrow[r, "i'"] \ar[equal]{d} & G_{\gamma\smin\alpha} C \arrow[r, "j'"] \ar[equal]{d} & G_{\gamma\smin\beta'} C \arrow[r] & 0 \\
            & G_{\gamma\smin\beta} C                                                & G_{\beta'\smin\alpha} C \oplus G_{\gamma\smin\beta'} C                                                &             &  
\end{tikzcd}
\end{equation}
Since $i,i'$ are chain homomorphisms by the assumption that the differential $\dff_C$ is $\sO(\sP)$-filtered we obtain the following splitting lemma for the differential $\dff_C$.
\begin{lemma}
    \label{splitdiff}
    Under the assumption that $\beta\smin\alpha = \gamma\smin\beta'$  and $\beta'\smin\alpha = \gamma\smin\beta$
    the differential $\dff_C\colon G_{\gamma\smin\alpha} C \to G_{\gamma\smin\alpha} C$ is given by:
    \begin{equation}
    \label{splitdiff2}
 \dff_{G_{\gamma\smin\alpha} C} = \begin{pmatrix} \dff_{G_{\beta\smin\alpha} C} & 0 \\ 0~ & \dff_{G_{\gamma\smin\beta} C}\end{pmatrix}\colon G_{\beta\smin\alpha} C \oplus G_{\gamma\smin\beta} \to G_{\beta\smin\alpha} C \oplus G_{\gamma\smin\beta},
\end{equation}
where $\dff_{G_{\beta\smin\alpha} C}$ and $\dff_{G_{\gamma\smin\beta} C}$ are the restrictions of $\dff_C$ to
$G_{\beta\smin\alpha} C$ and $G_{\gamma\smin\beta}$ respectively.
\end{lemma}
\begin{proof}
If we apply the first part of the proof of Lemma \ref{connhomchar12} to both exact sequences in \eqref{twoexactseq} we obtain that the off diagonal terms are zero completing the proof.
    \qed
\end{proof}
\begin{remark}
    For Cartan-Eilenberg systems generated by $\sP$-graded differential groups the conclusion of Lemma \ref{idforincomp} follows immediately since $ji'=\id$ and $j'i=\id$ as indicated in Diagram \eqref{twoexactseq}. The direct sum decomposition is immediate from the form of the differential.
\end{remark}

\section{Strict representations}
\label{strictrepres}
Two $\sO(\sP)$-filtered chain homomorphisms $f,g\colon (C,\dff_C)\to (A,\dff_A)$ are said to be { $\sO(\sP)$-filtered chain homotopic} if there is an $\sO(\sP)$-filtered homomorphism $h\colon C\to A$ such that $f-g = h \dff_C +\dff_A h$; such a map $h$ is called an {$\sO(\sP)$-filtered chain homotopy} from $f$ to $g$.  Notation: $f'\sim f$. 
In this case  $(C,\dff_C)$ and $(A,\dff_A)$ are \emph{$\sO(\sP)$-filtered chain homotopic}.
Notation: $(C,\dff_C) \sim (A,\dff_A)$.
 An $\sO(\sP)$-filtered chain homomorphism $f\colon (C,\dff_C)\to (A,\dff_A)$ is an {$\sO(\sP)$-filtered chain (homotopy) equivalence} if there is an $\sO(\sP)$-filtered chain homomorphism $f'\colon ( A,\dff_A)\to (C,\dff_C)$ such that $f' f\sim \id_C$ and $f f'\sim \id_{A}$. 
 In this case  $(C,\dff_C)$ and $(A,\dff_A)$ are \emph{$\sO(\sP)$-filtered chain equivalent}.
Notation: $(C,\dff_C) \simeq (A,\dff_A)$. 
Two differential groups  $(C,\dff_C)$ and $(A,\dff_A)$ are \emph{$\sO(\sP)$-filtered chain isomorpic} if there exist 
$\sO(\sP)$-filtered chain homomorphisms 
 $f\colon ( C,\dff_C)\to (A,\dff_A)$ and 
$f'\colon ( A,\dff_A)\to (C,\dff_C)$ such that $f'f=\id_C$ and $ff'=\id_A$.
In this case $f$ is called an $\sO(\sP)$-filtered chain isomorphism.
 Notation: $(C,\dff_C) \cong (A,\dff_A)$.

In this section we establish the fundamental equivalence result in the setting of strict, $\sP$-graded differential groups.
A $\sP$-graded group $C = \oplus_{p\in \sP} G_pC$ is free if and only if the groups $G_p C$ are free for all $p\in \sP$. 

\begin{theorem}\label{thm:equivalence}
Two 
free, finitely generated, strict, $\sP$-graded differential groups $(C,\dff_C)$ and $(A,\dff_A)$  are $\sO(\sP)$-filtered chain isomorphic
if and only if $\bfE(C,\dff_C) \cong \bfE(A,\dff_A)$. 
\end{theorem}

\begin{proof}%[of Theorem \ref{thm:equivalence}]
If $(C,\dff_C)$ and $(A,\dff_A)$ are $\sO(\sP)$-filtered chain isomorphic then the associated Cartan-Eilenberg systems are isomorphic which follows directly
from the definition of Cartan-Eilenberg system for a graded differential group. For the converse
we argue by induction using the lattice structure of $\sO(\sP)$. 

Let $q\in \sP$ be a minimal element and let $\alpha =\{q\}$. Since $\bfE(C,\dff_C)\cong \bfE(A,\dff_A)$ we have $F_\alpha C=G_q C$, $F_\alpha A=G_q A$, and a chain
isomorphism $f\colon G_q C\xrightarrow[]{\cong} G_q A$, with $f(F_\alpha C)=
F_\alpha A$,
which uses the fact that   $(C,\dff_C)$ and $(A,\dff_A)$ are finitely generated and strict, cf.\ Rmk.\ \ref{homequal}. Since the differentials on $G_qC$ and $G_qA$ vanish $H(f)=f$. This holds for all minimal elements $q\in \sP$, which concludes the first step
in the induction.
Let $\alpha\in\sO(\sP)$ and suppose $f'\colon F_\alpha C\to F_\alpha A$ is an $\sO(\sP)$-filtered chain isomorphism with
\begin{equation}
    \label{equalforiso}
f'(F_{\alpha'} C) = F_{\alpha'} A,\quad \forall \alpha'\subseteq \alpha.
\end{equation}
Let  $q\in \sP\smin \alpha$ be  a minimal element.
Define $\beta = \bigl\downarrow q$, which is a join-irreducible down-set in $\sO(\sP)$ with unique immediate predecessor $\beta^\dagger\subset \alpha$.
This implies that $\beta = \beta^\dagger \cup \{q\}$.
The next step is to construct a chain isomorphism $f \colon F_\beta C \xrightarrow[]{\cong}F_\beta A$, where $F_\beta C = F_{\beta^\dagger} C \oplus G_qC$ and $F_\beta A = F_{\beta^\dagger} A \oplus G_qA$, and which satisfies $f(F_{\beta} C) = F_{\beta} A$.
Since $\bfE(C,\dff_C)\cong \bfE(A,\dff_A)$ we have a chain
isomorphism $f''\colon G_q C\xrightarrow[]{\cong} G_q A$ and an isomorphism
$g\colon H(F_{\beta^\dagger} C\oplus G_qC) \to H(F_{\beta^\dagger} A\oplus G_qA)$ given by
the exact triangles:
\[
\begin{tikzcd}
             & H(F_{\beta^\dagger} C) \ar[equal]{d}           & H(F_{\beta^\dagger} C\oplus G_qC) \ar[equal]{d}           & G_qC \ar[equal]{d}           &                       &    \\
{} \arrow[r] & E^{\beta^\dagger}_\emptyset \arrow[d, "H(f')"'] \arrow[r]                & E^\beta_\emptyset \arrow[d, "g"] \arrow[r]                 & E^\beta_{\beta^\dagger} \arrow[d, "f''"] \arrow[r]                 & E^{\beta^\dagger}_\emptyset \arrow[r] \arrow[d] & {} \\
{} \arrow[r] & E'^{\beta^\dagger}_\emptyset \arrow[r] \ar[equal]{d} & E'^\beta_\emptyset \arrow[r] \ar[equal]{d} & E'^\beta_{\beta^\dagger} \arrow[r] \ar[equal]{d} & E'^{\beta^\dagger}_\emptyset \arrow[r]           & {} \\
             & H(F_{\beta^\dagger} A)                                          & H(F_{\beta^\dagger} A\oplus G_qA)                                         & G_qA                                          &                       &   
\end{tikzcd}
\]
Since $(C,\dff_C)$ and $(A,\dff_A)$ are finitely generated and strict we have 
$H(G_qC) =G_qC$ and $H(G_qA) = G_qA $. On the chain level this yields the following diagram:
\begin{equation}
    \label{ind5lem}
    \begin{tikzcd}
0 \arrow[r] & F_{\beta^\dagger} C \arrow[r] \arrow[d, "f'"'] & F_{\beta^\dagger} C\oplus G_qC \arrow[r] \arrow[d, dashed, "f"] & G_qC \arrow[r] \arrow[d, "f''"] & 0 \\
0 \arrow[r] & F_{\beta^\dagger} A \arrow[r]                 & F_{\beta^\dagger} A\oplus G_qA \arrow[r]                & G_qA \arrow[r]                & 0
\end{tikzcd}
\end{equation}
As before, since both $(C,\dff_C)$ and $(A,\dff_A)$ are finitely generated, strict and free  it follows that
 $G_q C$ and $G_q A$ are finitely generated, free, boundaryless abelian groups (and thus projective).
 The components $F_{\beta^\dagger} C$ and $F_{\beta^\dagger} A$ are free differential groups.
By Theorem \ref{lem:main} there exists a chain homomorphism $f$ that induces $g$. 
The chain homomorphism $f$ is a chain isomorphism via  the five lemma
since both $f'$ (by induction) and $f''$ are chain isomorphisms.  The isomorphism $f$ is filtered by construction, i.e.
$f(F_\beta C)\subset F_\beta A$. The above procedure yields a more precise statement, i.e. $\beta'\subsetneq \beta$ if and only if $\beta'\subset {\beta^\dagger}$, which implies, by the induction hypotheses, that $f(F_{\beta'}C)=f'(F_{\beta'}C)= F_{\beta'}A$ for all
$\beta'\subset \beta^\dagger\subset \alpha$.
By the form of $f$ given in \eqref{formoff}, and the fact that  $f''(G_qC)=G_qA$,  we have  
\[
\begin{aligned}
   f(F_{\beta} C) &=f\bigl(F_{\beta^\dagger} C\oplus G_qC\bigr) = \bigl[ f'(F_{\beta^\dagger}C) + \gamma(G_qC)\bigr]\oplus f''(G_qC) \\
   &= F_{\beta^\dagger}A \oplus G_qA = F_{\beta} A,
\end{aligned}
\]
since 
$\gamma(G_qC)\subset F_{\beta^\dagger}A$.

It remains to construct $f$ on $F_{\beta'}C$ 
for down-sets $\beta'\subset \alpha\cup\{q\}$, where $\alpha\cup\{q\}\in \sO(\sP)$. 
Since every down-set can be uniquely represented as an irredundant union of join-irreducible down-sets,  the down-set $\beta'$ is represented as
\begin{equation}
    \label{irrjoin}
    \beta' = \bigl\downarrow q \cup \bigl\downarrow p_{i_1} \cup \cdots \cup \bigl\downarrow p_{i_k}
    =\beta\cup \gamma, \quad p_{i_1},\cdots,p_{i_k}\in \alpha,\quad \beta=\bigl\downarrow q,
\end{equation}
and  $\gamma = \bigl\downarrow p_{i_1} \cup \cdots \cup \bigl\downarrow p_{i_k}\subset \alpha$.
For $\beta\cup \gamma$ and $\beta\cap \gamma$ we have the following diagram:
\[
\begin{tikzcd}[cramped, sep=small]%[column sep=tiny]
             & (\beta\cap \gamma)\cup (\beta\smin\gamma) \cup (\gamma\smin\beta) \arrow[ld] \arrow[rd] &              \\
(\beta\cap \gamma)\cup (\beta\smin\gamma)  \arrow[rd] &                         & (\beta\cap \gamma)\cup (\gamma\smin\beta)  \arrow[ld] \\
             & \beta\cap \gamma                       &             
\end{tikzcd}
\]
where $(\beta\cap \gamma)\cup (\beta\smin\gamma) \cup (\gamma\smin\beta)=\beta\cup \gamma = \beta'$.
Consider the decompositions:
\[
F_\beta C= F_{\beta\cap\gamma} C\oplus G_{\beta\smin\gamma} C,\quad
F_\gamma C = F_{\beta\cap\gamma} C\oplus G_{\gamma\smin\beta} C
\]
and 
\[
F_{\beta'} C= F_{\beta\cap\gamma} C\oplus \Bigl[ G_{\beta\smin\gamma} C \oplus G_{\gamma\smin\beta} C \Bigr] = F_\beta C + F_\gamma C.
\]
The same decomposition holds for $F_\beta A$, $F_\gamma A$ and $F_{\beta'} A$.
Since $\beta\smin\gamma$ and $\gamma\smin\beta$ are uncomparable convex sets Lemma \ref{splitdiff}, with
the triples $(\beta\cap\gamma,\beta,\beta\cup\gamma)$ and $(\beta\cap\gamma,\gamma,\beta\cup\gamma)$, implies that 
 the differential on $G_{\beta\smin\gamma} C \oplus G_{\gamma\smin\beta} C$ is a direct sum and thus
the differentials are of the form:
\begin{equation}
    \label{newdiffs12}
\dff_{F_\beta C} = \begin{pmatrix} \dff_{F_{\beta\cap \gamma}C}  & \mathrm{\Lambda} \\ 0 & \dff_{G_{\beta\smin \gamma}C} \end{pmatrix}, 
\quad \dff_{F_\gamma C}=\begin{pmatrix} \dff_{F_{\beta\cap \gamma}C}  & \mathrm{\Lambda'} \\ 0 & \dff_{G_{\gamma\smin \beta}C} \end{pmatrix}, 
\end{equation}
and
\begin{equation}
    \label{newdiffs34}
    \dff_{F_{\beta'} C} = \begin{pmatrix} \dff_{F_{\beta\cap \gamma}C}  & \mathrm{\Lambda}~~ & \mathrm{\Lambda'}\\ 0 & \dff_{G_{\beta\smin \gamma}C}~~ &0 \\ 0 & 0~~ & \dff_{G_{\gamma\smin \beta}C} \end{pmatrix},
\end{equation}
It is readily verified that with $C= F_{\beta} C$, $C'= F_\gamma C$, $A = F_{\beta} A$ and $A' = F_\gamma A$, and 
$C+C'=F_{\beta'}C$ and $A+A'=F_{\beta'}A$, Theorem \ref{lem:main22}(i) is satisfied for the differentials.
By construction the chain isomorphisms $f'=f|_{F_{\beta} C}$ and $f''=f|_{F_{\gamma} C}$ satisfy Theorem \ref{lem:main22}(ii) and therefore $f := f|_{F_\beta C+F_\gamma C}$
is a well-defined chain homomorphism on $F_{\beta'} C$ and 
$f(F_{\beta'}C)=F_{\beta'}A$.
Indeed, $f'(F_\beta C)=F_\beta A$ and $f''(F_\gamma C) = F_\gamma A$, and
$f(c+c') = f'(c)+f''(c')$ for $c\in F_\beta C$ and $c'\in F_\gamma C$.
This implies that 
\[
f(F_{\beta'}C) = f(F_\beta C+F_\gamma C) = F_\beta A+ F_\gamma A = F_{\beta'}A.
\]
and therefore 
the chain homomorphism is well-defined on $F_{\alpha\cup \downarrow q}C$ and is filtered, i.e. 
$f(F_{\alpha'} C)= F_{\alpha'}A$ for all $\alpha' \subset \bigl\downarrow q \cup\alpha$. 
Since $f(F_{\beta\cap \gamma} C)= F_{\beta\cap\gamma}A$, the homomorphism $f$ is of the form
\[
f'=f|_{F_\beta C}= \begin{pmatrix} f|_{F_{\beta\cap \gamma}C}  & \mathrm{\chi} \\ 0 & f|_{G_{\beta\smin \gamma}C} \end{pmatrix},\quad
f''=f|_{F_\gamma C}= \begin{pmatrix} f|_{F_{\beta\cap \gamma}C}  & \mathrm{\chi'} \\ 0 & f|_{G_{\gamma\smin \beta}C} \end{pmatrix}
\]
Since $f|_{F_\beta C}$ and $f|_{F_\gamma C}$ are isomorphisms  so are $f|_{G_{\beta\smin\gamma}C}$ and $f|_{G_{\gamma\smin\beta}C}$, and by construction $f(G_{\beta\smin\gamma}C)\subset G_{\beta\smin\gamma}A$ and
$f(G_{\gamma\smin\beta}C)\subset G_{\gamma\smin\beta}A$.
Since $\beta'\smin\beta = \gamma\smin\beta$ and $\beta'\smin\gamma = \beta\smin\gamma$ we have the exact triangles:
\begin{equation}
    \label{twoexacttr12}
\begin{tikzcd}[column sep=small]
E^\beta_\varnothing \arrow[rr, "i"] &                   & E^{\beta'}_\varnothing \arrow[ld, "j"] \\
                  & E^{\beta'}_\beta \cong E^\gamma_\beta \arrow[lu, "k"] &                  
\end{tikzcd}
\qquad\qquad
\begin{tikzcd}[column sep=small]
E^\gamma_\varnothing \arrow[rr, "i"] &                   & E^{\beta'}_\varnothing \arrow[ld, "j"] \\
                  & E^{\beta'}_\gamma \cong E^\beta_\gamma \arrow[lu, "k"] &                  
\end{tikzcd}
\end{equation}
From the isomorphism between Cartan-Eilenberg systems we obtain the commutative squares:
\[
\begin{tikzcd}[column sep=large, row sep=large]
E^\beta_\varnothing \arrow[r, "i"] \arrow[d, "H(f')"'] & E^{\beta'}_\varnothing \arrow[d, "g"] \\
E'^\beta_\varnothing \arrow[r, "i"]                 & E'^{\beta'}_\varnothing               
\end{tikzcd}
\quad\quad
\begin{tikzcd}[column sep=large, row sep=large]
E^\gamma_\varnothing \arrow[r, "i"] \arrow[d, "H(f'')"'] & E^{\beta'}_\varnothing \arrow[d, "g"] \\
E'^\gamma_\varnothing \arrow[r, "i"]                 & E'^{\beta'}_\varnothing               
\end{tikzcd}
\]
By Theorem \ref{lem:main22}(iii) $f$ lifts $g$ and since $f'$ and $f''$ are chain isomorphisms by construction the five lemma applied to the short exact sequences associated to the exact triangles in \eqref{twoexacttr12} imply that also $f$ is a chain isomorphism, i.e.
\[
    \begin{tikzcd}
0 \arrow[r] & F_{\beta} C \arrow[r] \arrow[d,"\cong", "f'"'] & F_{\beta} C + F_\gamma C \arrow[r] \arrow[d, dashed, "f"] & G_{\gamma\smin\beta} C \arrow[r] \arrow[d, "\cong"] & 0 \\
0 \arrow[r] & F_{\beta} A \arrow[r]                 & F_{\beta} A + F_\gamma A \arrow[r]                & G_{\gamma\smin\beta} A \arrow[r]                & 0
\end{tikzcd}
\]
and similarly for 
\[
    \begin{tikzcd}
0 \arrow[r] & F_{\gamma} C \arrow[r] \arrow[d,"\cong", "f''"'] & F_{\beta} C + F_\gamma C \arrow[r] \arrow[d, dashed, "f"] & G_{\beta\smin\gamma} C \arrow[r] \arrow[d, "\cong"] & 0 \\
0 \arrow[r] & F_{\gamma} A \arrow[r]                 & F_{\beta} A + F_\gamma A \arrow[r]                & G_{\beta\smin\gamma} A \arrow[r]                & 0
\end{tikzcd}
\]
which completes 
 the construction of $f$ on $F_{\beta'}C$ 
for down-sets $\beta'\subset \alpha\cup\{q\}$. This process terminates after finitely many steps.
\qed
\end{proof}

\begin{remark}
    If $h$ is a morphism between $\bfE$ and $\bfE'$ such that $h|_{E^\beta_\alpha}$ is an isomorphism for all $\beta\smin\alpha=\{p\}$, $p\in \sP$, then $h$ is a isomorphism on all $E$-terms $E^\beta_\alpha$, for all 
    $\alpha\subset \beta$. 
    This is a direct consequence of the five lemma.
\end{remark}

\section{Equivalence to Franzosa's module braids}
\label{franmodbr}
In \cite{fran}  a data structure is developed to relate Conley indices of Morse sets in a Morse representation, cf.\ \cite{lsa3}.
In this section we show that Franzosa's data structure is equivalent to a Cartan-Eilenberg system.

\subsection{Module braids}
\label{modbr1}
The starting point in \cite{fran} is a finite poset $(\sP,\le)$. 
We recall some notation from
 \cite{fran} to relate elements in $\sCo(\sP)$, the meet-semilattice of convex sets in $\sP$, and formulate these concepts in terms of lattice theory.
Two elements $\xi,\eta\in \sCo(\sP)$  are \emph{adjacent} if and only if there exists an ordered  triple $(\alpha,\beta,\gamma)$ such that
    $\xi=\beta\smin\alpha$ and $\eta = \gamma\smin\beta$. Then,  $\xi\cap\eta=0$ and the union satisfies $\xi\eta:=\xi\cup\eta = \gamma\smin\alpha$ which is a \emph{decomposition} of convex sets.
For every convex set $\xi\in \sCo(\sP)$ assign  a differential  group $(D_\xi,\dff_\xi)$ such that 
\[
\begin{tikzcd}
0 \arrow[r] & \displaystyle{D_\xi} \arrow[r, "i"] & \displaystyle{D_{\xi\eta}} \arrow[r, "j"] & \displaystyle{D_{\eta}} \arrow[r] & 0
\end{tikzcd},
\]
is exact.\footnote{In \cite{fran} the slightly more general assumption of weak exactness is assumed. 
The results in this section remain true under the assumption of weak exactness.
Weak exactness in \cite{fran} can be avoided by using Borel-Moore homology for example.}
If $\xi$ and $\eta$ are incomparable then $j'i=\id$ on $_\xi$ where $i\colon D_\xi \to D_{\xi\eta}$ and $j'\colon D_{\eta\xi} \to D_\xi$. 
Franzosa refers to the above structure as a `chain complex braid'.
The prime example of a chain complex braid is given by a $\sO(\sP)$-filtered differential group $(D,\dff_D)$ and is given in \eqref{firstexactseq13}.
As in Section \ref{relexsys} one proves that $G_{\beta\smin\alpha}D\cong G_{\beta'\smin\alpha'} D$ for $\beta\smin\alpha = \beta'\smin\alpha' =:\xi$ along with the other properties. 
The homology is denote by $E_\xi:= H(D_\xi,\dff_\xi)$ and satisfies the following properties:
\begin{definition}
\label{maindefnmodbr}
A  \emph{module braid} $\bfB$ over a meet-semilattice $\sCo(\sP)$ is a collection of abelian groups (modules) $E_\xi\in \sRmod$, indexed by $\xi\in \sCo(\sP)$, such that
\begin{enumerate}
    \item [(i)] for every pair of \emph{adjacent}
    $\xi,\eta\in \sCo(\sP)$ the triangles
\begin{equation}
    \label{braid111}
    \begin{tikzcd}[column sep=0.5cm]
E_{\xi} \arrow[rr, "i"] && E_{\xi\eta} \arrow[ld, "j"] \\
& E_{\eta} \arrow[lu, "k"] &
\end{tikzcd}
    \end{equation}
are exact;
    \item[(ii)] for \emph{incomparable} elements $\xi,\eta\in \sCo(\sP)$, cf.\ Rmk.\ \ref{convexpara}, 
    both triangles
    \begin{equation}
    \label{braid111again}
    \begin{tikzcd}[column sep=0.5cm]
E_{\xi} \arrow[rr, "i"] && E_{\xi\eta} \arrow[ld, "j"] \\
& E_{\eta} \arrow[lu, "k"] &
\end{tikzcd}
\qquad
\begin{tikzcd}[column sep=0.5cm]
E_{\eta} \arrow[rr, "i'"] && E_{\eta\xi} \arrow[ld, "j'"] \\
& E_{\xi} \arrow[lu, "k'"] &
\end{tikzcd}
    \end{equation}
are exact and
    $j'  i= \id$ on $E_\xi$ and $j i' = \id$ on $E_\eta$.
    \item[(iii)] for every triple of adjacent elements\footnote{A triple $\xi,\eta,\zeta$ of adjacent convex sets is equivalent to a  ordered quadruple $(\alpha,\beta,\gamma,\delta)$ such that
    $\xi=\beta\smin\alpha$, $\eta = \gamma\smin\beta$ and $\zeta = \delta\smin\gamma$. These convex sets are mutually disjoint and their union is $\delta\smin\alpha$.} $\xi,\eta,\zeta\in \sCo(\sP)$ the `braid diagram' 
\begin{equation}
    \label{braiddiag1234}
    \begin{tikzcd}[cramped, sep=small]%[column sep=tiny]
{} \arrow[d, bend right, shift right=3]                    &  & {} \arrow[lld] \arrow[rrd, dashed] &  & {} \arrow[d, bend left, shift left=3]           \\
E_\xi \arrow[rrd, "i"] \arrow[dd, "i"', bend right, shift right=3]     &  &                                            &  & E_\zeta \arrow[lld, "k"'] \arrow[dd, "k", dashed, bend left, shift left=2]     \\
                                                                   &  & E_{\xi\eta} \arrow[rrd, "j"] \arrow[lld, "i"']       &  &                                                                  \\
E_{\xi\eta\zeta} \arrow[rrd, "j"] \arrow[dd, "j"', bend right, shift right=3]     &  &                                            &  & E_\eta \arrow[lld, dashed, "i"'] \arrow[dd, "k", bend left, shift left=2]     \\
                                                                   &  & E_{\eta\zeta} \arrow[lld, dashed, "j"'] \arrow[rrd, "k"]       &  &                                                                  \\
E_\zeta \arrow[rrd, "k"] \arrow[dd, dashed, "k"', bend right, shift right=3]     &  &                                            &  & E_\xi \arrow[lld, "i"'] \arrow[dd, "i", bend left, shift left=2]     \\
                                                                   &  & E_{\xi\eta} \arrow[lld, "j"'] \arrow[rrd, "i"]       &  &                                                                  \\
E_\eta \arrow[d, bend right, shift right=4] \arrow[rrd, dashed] &  &                                            &  & E_{\xi\eta\zeta} \arrow[d, bend left, shift left=3] \arrow[lld] \\
{}                                                                 &  & {}                                         &  & {}                                                              
\end{tikzcd}
    \end{equation}
commutes. 
\end{enumerate}
\end{definition}

\begin{remark}
    \label{braid1}
    The terminology module braid is best explained since   the exact triangles in \eqref{braid111} correspond to the strands in the braid diagram in \eqref{braiddiag1234}. 
    In the spirit of triangulated categories the braid diagram can also be reshaped as a commutative diagram of exact triangles: 
\begin{equation}
\label{exact9yy}
\begin{tikzcd}[column sep=small]
                                                 &                                    & E_\xi \arrow[lldd, "i"', bend right] \arrow[rrdd, "i", bend left] &                                     &                                    \\
                                                 & E_\eta \arrow[ru, "k"] \arrow[rr, dashed, "i"'] &                                                               & E_{\eta\zeta} \arrow[lu, "k"'] \arrow[ld, dashed, "j"'] &                                    \\
E_{\xi\eta} \arrow[ru, "j"] \arrow[rrrr, "i"', bend right] &                                    & E_{\zeta} \arrow[ll, "k"] \arrow[lu, dashed, "k"']                            &                                     & E_{\xi\eta\zeta} \arrow[ll, "j"] \arrow[lu, "j"']
\end{tikzcd}
\end{equation}
The dashed `strand' corresponds to the middle  exact triangle in \eqref{exact9yy}, which makes the diagram commutative.
\end{remark}

\begin{lemma}
    \label{tooctodiag}
    Let $\bfE$ be a Cartan-Eilenberg system over $\sO(\sP)$, with $\sP$ a finite poset. Then, the  diagram 
    \begin{equation}
  \label{octadiag12}
  \vspace{-0.2cm}
\begin{tikzcd}[column sep=small]
                                                 &                                    & E^{\beta}_{\alpha} \arrow[lldd, "i"', bend right] \arrow[rrdd, "i", bend left] &                                     &                                    \\
                                                 & E^{\gamma}_{\beta} \arrow[ru, "k"] \arrow[rr, "i"',dashed] &                                                               & E^{\delta}_{\beta} \arrow[lu, "k"'] \arrow[ld, "j"',dashed] &                                    \\
E^{\gamma}_{\alpha} \arrow[ru, "j"] \arrow[rrrr, "i"', bend right] &                                    & E^{\delta}_{\gamma} \arrow[ll, "k"] \arrow[lu, "k"',dashed]                            &                                     & E^{\delta}_{\alpha} \arrow[ll, "j"] \arrow[lu, "j"']
\end{tikzcd}
\end{equation}
commutes for every ordered quadruple $(\alpha,\beta,\gamma,\delta)$ in $\sO(\sP)$.
\end{lemma}
\begin{proof}
The commutative diagram follows by defining the dashed triangle by considering the compositions: $E^\gamma_\beta \xrightarrow[]{k}E^\beta_\alpha\xrightarrow[]{i}E^\delta_\alpha\xrightarrow[]{j} E^\delta_\beta$, $E^\delta_\beta\xrightarrow[]{k}E^\beta_\alpha\xrightarrow[]{i}E^\delta_\alpha\xrightarrow[]{j} E^\delta_\gamma$, and $E^\delta_\gamma\xrightarrow[]{k}E^\gamma_\alpha\xrightarrow[]{j}E^\gamma_\beta$. 
The first two compositions are equal to the homomorphisms
$i\colon E^\gamma_\beta \to E^\delta_\beta$ and $j\colon E^\delta_\beta\to E^\delta_\gamma$,
since nested ordered pairs correspond to unique homomorphisms.
The latter composition is the connecting homomorphism $k\colon E^\delta_\gamma\to E^\gamma_\beta$ due to the commutative squares for $k$ in \eqref{unpackeddiag12}[right]. This implies that the dashed triangle is  the exact triangle for  $(\beta,\gamma,\delta)$.
    \qed
\end{proof}

The  diagram in \eqref{octadiag12} is a reshaped version of the octahedral diagram in triangulated categories and is a structure present in Cartan-Eilenberg systems, cf.\ \cite[Lem.\ 4.8]{Matschke}. The octahedral diagram can also be used in the definition of a Cartan-Eilenberg system by removing the commutative squares in \eqref{unpackeddiag12}[right]
and adding some additional conditions.
This idea lies at the heart of proving that Franzosa's module braids correspond to excisive Cartan-Eilenberg systems.

\begin{theorem}
\label{equivEEC}
A  module braid $\bfB$  over a meet-semilattice $\sCo(\sP)$ for some finite poset $\sP$ is equivalent to
an excisive Cartan-Eilenberg system  $\bfE$ over $\sO(\sP)$.
\end{theorem}
\begin{proof}
    We start with the observation that an excisive Cartan-Eilenberg system we choose $E$-terms over $\sCo(\sP)$ by setting $E_\xi=E^\beta_\alpha\cong E^{\beta'}_{\alpha'}$, with $\xi=\beta\smin\alpha=\beta'\smin \alpha'$. The exact triangles for $\bfE$ yield the exact triangle for $\bfB$ in
     (i). Property (ii) follows from Lemma \ref{idforincomp} and the octahedral diagram (module braid diagram) is given by Lemma \ref{tooctodiag},
which establishes Property (iii). %, cf. \cite[Lemma 4.8]{Matschke}. 
We conclude that an excisive Cartan-Eilenberg system $\sO(\sP)$ defines a module braid over $\sCo(\sP)$.

For the converse we want to show that (i)-(iii) in the definition of module braid define an excisive Cartan-Eilenberg system in the sense that the commutative squares in \eqref{unpackeddiag12}[right] exist.
Define $E^\beta_\alpha := E_\xi$ for all ordered  pairs $(\alpha,\beta)$ such that $\beta\smin \alpha = \xi$. In \eqref{braid111} choose  $\xi=\alpha$ and $\eta = \beta\smin \alpha$ for any the ordered pair $(\alpha,\beta)$, which establishes the associated exact triangles in \eqref{unpackeddiag12}[left], together with the maps $i$, $j$ and $k$.
 The next step is to define the maps 
 $\ell\colon E^\beta_\alpha \to E^{\beta'}_{\alpha'}$ and
 $\ell\colon E^\gamma_\beta \to E^{\gamma'}_{\beta'}$ in \eqref{unpackeddiag12}[right]
for ordered triples $(\alpha,\beta,\gamma)\le(\alpha',\beta',\gamma')$.
Consider the octahedral diagrams in \eqref{exact9yy} in the following two configurations:\footnote{The notation $\genfrac{}{}{0pt}{}{\alpha}{~}
\genfrac{}{}{0pt}{}{\undergroup{\subset}}{\eta}
\genfrac{}{}{0pt}{}{\beta}{ ~}$ indicates the convex set $\eta=\beta\smin\alpha$.}
\begin{equation}
    \label{commdiagconstr}
\begin{aligned}
&\genfrac{}{}{0pt}{}{\alpha}{~}
\genfrac{}{}{0pt}{}{\undergroup{\subset}}{\xi}
\genfrac{}{}{0pt}{}{\beta}{~}
\genfrac{}{}{0pt}{}{\undergroup{\subset}}{\eta}
\genfrac{}{}{0pt}{}{\gamma}{ ~}
\genfrac{}{}{0pt}{}{\undergroup{\subset}}{\zeta}
\genfrac{}{}{0pt}{}{\gamma'}{ ~} \\
&\genfrac{}{}{0pt}{}{\alpha}{ ~}
\genfrac{}{}{0pt}{}{\undergroup{\subset}}{\xi}
\genfrac{}{}{0pt}{}{\beta}{ ~}
\genfrac{}{}{0pt}{}{\undergroup{\subset}}{\eta'}
\genfrac{}{}{0pt}{}{\beta'}{ ~}
\genfrac{}{}{0pt}{}{\undergroup{\subset}}{\zeta'}
\genfrac{}{}{0pt}{}{\gamma'}{ ~}
\end{aligned}
\quad\quad
\begin{tikzcd}
E_\eta \arrow[r, "k_1"] \arrow[d, "i"'] & E_\xi \arrow[dd, "i'_1"] \\
E_{\eta\zeta}=E_{\eta'\zeta'} \arrow[ur, "k_2","k_2'"'] \arrow[d, "j'"'] & \\
E_{\zeta'} \arrow[r, "k'_3"] & E_{\xi\eta'}
\end{tikzcd}
\end{equation}
Diagram \eqref{exact9zz} below gives the octahedral diagrams for both configurations.
By using the first configuration Diagram \eqref{exact9zz} yields the upper part of Diagram \eqref{commdiagconstr}[right] and by using the second configuration in \eqref{exact9zz} we obtain the lower part of  Diagram \eqref{commdiagconstr}[right]. 
\begin{equation}
\label{exact9zz}
\begin{tikzcd}[column sep=small]
                                                 &                                    & E_\xi \arrow[lldd, "i_1"', bend right] \arrow[rrdd, "i_2", bend left] &                                     &                                    \\
                                                 & E_\eta \arrow[ru, "k_1"{near start}] \arrow[rr, dashed, "i"'] &                                                               & E_{\eta\zeta} \arrow[lu, "k_2"'{near start}] \arrow[ld, dashed, "j"'] &                                    \\
E_{\xi\eta} \arrow[ru, "j_1"{near end}] \arrow[rrrr, "i_3"', bend right] &                                    & E_{\zeta} \arrow[ll, "k_3"] \arrow[lu, dashed, "k"']                            &                                     & E_{\xi\eta\zeta} \arrow[ll, "j_3"] \arrow[lu, "j_2"'{near end}]
\end{tikzcd}
\!\!\!%\quad
\begin{tikzcd}[column sep=small]
                                                 &                                    & E_{\xi} \arrow[lldd, "i'_1"', bend right] \arrow[rrdd, "i'_2", bend left] &                                     &                                    \\
                                                 & E_{\eta'}
                                                 \arrow[ru, "k'_1"{near start}] \arrow[rr, dashed, "i'"'] &                                                               & E_{\eta'\zeta'} \arrow[lu, "k'_2"'{near start}] \arrow[ld, dashed, "j'"'] &                                    \\
E_{\xi\eta'} \arrow[ru, "j'_1"{near end}] \arrow[rrrr, "i'_3"', bend right] &                                    & E_{\zeta'} \arrow[ll, "k'_3"] \arrow[lu, dashed, "k'"']                            &                                     & E_{\xi\eta'\zeta'} \arrow[ll, "j'_3"] \arrow[lu, "j'_2"'{near end}]
\end{tikzcd}
\end{equation}
Since the upper and lower diagrams in \eqref{commdiagconstr}[right]
commute the composition is also a commutative diagram.
This construction defines a homomorphism $\ell\colon E^\gamma_\beta \to E^{\gamma'}_{\beta'}$ for every
 ordered pair $(\beta,\gamma)\le(\beta',\gamma')$.
The composition properties for the maps $\ell$ follows by building adjacent diagrams from \eqref{commdiagconstr}[right], which
yields $E^\beta_\alpha \xrightarrow[]{\ell}E^{\beta'}_{\alpha'}\xrightarrow[]{\ell'}E^{\beta''}_{\alpha''}$.
Diagram \eqref{commdiagconstr}[right] together with composition yields  the following commutative diagrams:
\[
\begin{tikzcd}[column sep=1.4cm, row sep=1.0cm]
E^\gamma_\beta \arrow[d, "j'i"'] \arrow[r, "k_1"] & E^\beta_\alpha \arrow[r, "\id"] \arrow[d, "i'_1"] & E^\beta_\alpha \arrow[d, "\ell"]   & E^{\gamma'}_{\beta'} \arrow[d, "\id"'] \arrow[r, "k'_3"] \arrow[rd, "k'_3"] & E^{\beta'}_\alpha \arrow[d, "\ell'"] \\
E^{\gamma'}_{\beta'} \arrow[r, "k'_3"]                 & E^{\beta'}_{\alpha} \arrow[r, "\ell'"]                & E^{\beta'}_{\alpha'}                  & E^{\gamma'}_{\beta'} \arrow[r, "k'_3 j"]                                 & E^{\beta'}_{\alpha'}               
\end{tikzcd}
\]
which establishes  Diagram \eqref{unpackeddiag12}[right]. 
It remains to show that 
$\ell\colon E^{\beta}_{\alpha}\to E^{\beta'}_{\alpha'}$ is the identity map whenever $\beta\smin \alpha = \beta'\smin \alpha'$
and $(\alpha,\beta)\le (\alpha',\beta')$.
Then, the ordered triples $(\alpha,\beta,\beta')$ and $(\alpha,\alpha',\beta')$ satisfy the conditions in  Lemma \ref{idforincomp} which provides the exact triangles (ii) of the definition of module braid
\begin{equation}
    \label{exact3def}
    \begin{tikzcd}[column sep=0.5cm]
E_{\xi} \arrow[rr, "i"] && E_{\xi\eta} \arrow[ld, "j"] \\
& E_{\eta} \arrow[lu, "k"] &
\end{tikzcd}
\qquad
\begin{tikzcd}[column sep=0.5cm]
E_{\eta} \arrow[rr, "i'"] && E_{\eta\xi} \arrow[ld, "j'"] \\
& E_{\xi} \arrow[lu, "k'"] &
\end{tikzcd}
\end{equation}
with $\xi=\beta\smin\alpha=\beta'\smin\alpha'$ and
$\eta = \alpha'\smin\alpha=\beta'\smin\beta$ incomparable convex sets, and $\xi\eta=\beta'\smin\alpha$.
The homomorphism $\ell\colon E^{\beta}_{\alpha}\to E^{\beta'}_{\alpha'}$ is given by the composition
$E^\beta_\alpha \xrightarrow[]{i} E^{\beta'}_\alpha \xrightarrow[]{j'} E^{\beta'}_{\alpha'}$.
By Definition \ref{maindefnmodbr}(ii)  we have that $j'i=\id$
which proves that $\ell=\id$.
The $E$-terms constructed in this proof satisfy the hypotheses of a Cartan-Eilenberg system with the isomorphisms of excisive pairs being the identity isomorphism, i.e. $E^{\alpha\cup\beta}_\alpha=E^\beta_{\alpha\cap\beta}$
for all $\alpha,\beta\in \sO(\sP)$.
\qed
\end{proof}

\subsection{Connection matrices}
\label{connmat1}
In \cite{fran} a powerful representation result is established for module braids that are generated by a chain complex braid. 
Due to Theorem \ref{equivEEC} we can reformulate this result for Cartan-Eilenberg systems generated by an $\sO(\sP)$-filtered differential group.
Let $(D,\dff_D)$ be an $\sO(\sP)$-filtered differential group and let $\bfE(D,\dff_D)$ be the associated Cartan-Eilenberg system.
We say that the latter is free and finitely generated if the groups $E_\alpha^\beta$ are free and finitely generated for each pair $(\alpha,\beta)$ with $\beta\smin\alpha=\{p\}$ for all $p\in \sP$.
Define a free, finitely generated $\sP$-graded abelian group $A$ with the decomposition
\begin{equation}
    \label{deriveddiffgr}
A= \bigoplus_{p\in \sP} G_p A,\quad G_p A\cong E_\alpha^\beta,~~~\beta\smin\alpha=\{p\}.
\end{equation}
An appropriate differential is constructed by the following theorem:
\begin{theorem}[cf.\ \cite{fran}, Thm.\ 4.8]
    \label{mainresfran1}
    Let $(D,\dff_D)$ be an $\sO(\sP)$-filtered differential group with the property that
     the associated Cartan-Eilenberg system   $\bfE(D,\dff_D)$ is free  and finitely generated.
    Let $A$ be defined as in \eqref{deriveddiffgr}. Then, there exists an $\sO(\sP)$-filtered differential
    $\dff_A\colon A\to A$ such that 
    \[
    \bfE(A,\dff_A) \cong \bfE(D,\dff_D).
    \]
    In particular,
     there exists
    an $\sO(\sP)$-filtered quasi-isomorphism\footnote{A quasi-isomorphism is a chain homomorphism that is an isomorphism on homology.} $f\colon (A,\dff_A)\to (D,\dff_D)$,
    such that the isomorphisms $H(f)\colon E_\alpha^\beta(A,\dff_A) \to E_\alpha^\beta(D,\dff_D)$, for all $\alpha\subset \beta$, define an isomorphism of Cartan-Eilenberg systems. 
\end{theorem}

In \cite{fran} the strict, $\sP$-graded differential group $(A,\dff_A)$ is called a \emph{connection matrix} for $\bfE(D,\dff_D)$.
Note that $\dff_A$ restricted to the groups $G_pA$ is the zero map 
and therefore $H(G_pA)=G_pA$.
The differential may be regarded as a strict, upper-triangular matrix on $A$. In dynamical systems theory the non-trivial entries of $\dff_A$ contain information about heterclinic orbits between Morse sets.
In \cite{fran} the notion of connection matrix is defined in a slightly more general context. If we translate this to Cartan-Eilenberg systems we obtain:
\begin{definition}[\cite{fran}, Defn.\ 3.6.B]
    \label{connmat2}
    Let $\bfE$ be a Cartan-Eilenberg system over $\sO(\sP)$. A strict, $\sP$-graded differential group $(A,\dff_A)$ as defined \eqref{deriveddiffgr}, such that
    $\bfE(A,\dff_A) \cong \bfE$ is called a connection matrix for $\bfE$.
\end{definition}

\begin{remark}
    The above theorem is a translation of the more general result in \cite[Thm.\ 4.8]{fran} to the case of Cartan-Eilenberg systems generated by filtered differential groups. Existence of connection matrices is only established in the setting of \cite[Thm.\ 4.8]{fran}, i.e. for Cartan-Eilenberg systems $\bfE(D,\dff_D)$ in which case a connection matrix is quasi-isomorphic to $(D,\dff_D)$.
\end{remark}

\section{Discussion}
\label{discussion}
The main result of this paper in Theorem~\ref{thm:equivalence} has various implications for the theory of connection matrices as used in dynamical systems theory. In this section we discuss the applications and resolutions of the conjectures by Robbin and Salamon, as well as the implication for the transition matrices by Franzosa and Mischaikow, cf.\ \cite{atm}.

\subsection{The conjectures by J. Robbin and D. Salamon}
\label{robsal}
In their seminal paper on dynamical systems and connection matrices, cf.\ \cite{robbin:salamon2}, Robbin and Salamon formulate a number of conjectures concerning non-uniqueness of connection matrices.
Theorem~\ref{thm:equivalence} is posed as Conjecture 8.5 in \cite{robbin:salamon2}
in the setting of flows. Theorem \ref{thm:equivalence} extends the original conjecture by replacing field coefficients by ring coefficients ($R$ is a principal ideal domain). 
Theorem \ref{thm:equivalence}  resolves non-uniqueness issues in Franzosa's connection matrix theory (see \cite[Section 6.3]{fran}): all connection matrices are $\sO(\sP)$-filtered chain isomorphic, i.e.
all  connection matrices according to Definition \ref{connmat2} are conjugated via $\sO(\sP)$-filtered chain isomorphisms, which 
 generalizes the main result in \cite[Theorem 3.5]{atm}.

In \cite[Thm.\ 8.1]{robbin:salamon2}  it is proven  
that an $\sO(\sP)$-filtered differential vector space $(D,\dff_D)$  
 is $\sO(\sP)$-filtered chain  equivalent to a strict, $\sP$-graded differential vector space $(A,\dff_A)$, such that
 $\bfE(D,\dff_D)\cong \bfE(A,\dff_A)$.
In this setting our main result solves the following conjecture:
\begin{corollary}[cf.\ \cite{robbin:salamon2}, Conj.\ 7.4]
\label{conj74}
Let $(D,\dff_D)$ and $(D',\dff_{D'})$ be $\sO(\sP)$-filtered differential vector spaces  over a field $\F$.  Then, $\bfE(D,\dff_D)\cong \bfE(D',\dff_{D'})$ if and only if $(D,\dff_D)$ and $(D',\dff_{D'})$  are $\sO(\sP)$-filtered chain equivalent, i.e. $(D,\dff_D) \simeq(D',\dff_{D'})$. 
\end{corollary}
\begin{proof}
If $(D,\dff_D)$ and $(D',\dff_{D'})$  are $\sO(\sP)$-chain equivalent then the equivalence $\bfE(D,\dff_D)\cong \bfE(D',\dff_{D'})$ is immediate.
For vector spaces
the existence theorems in
\cite{robbin:salamon2,hms} show that there exists strict, 
$\sP$-graded differential vector spaces $(A,\dff_A)$ and $(C,\dff_C)$, and chain equivalences 
  $(A,\dff_A)\simeq (D,\dff_D)$ and $(C,\dff_C)\simeq (D',\dff_{D'})$, which implies that
$\bfE(D,\dff_D)\cong \bfE(A,\dff_A)$
and $\bfE(D',\dff_{D'})\cong \bfE(C,\dff_C)$. %, i.e. $(V,\delta_V)$
By assumption,
$\bfE(D,\dff_D)\cong \bfE(D',\dff_{D'})$ and thus
 $\bfE(A,\dff_A) \cong \bfE(C,\dff_C)$.  Consequently,  $(A,\dff_A) \cong (C,\dff_C)$ are $\sO(\sP)$-chain isomorphic via Theorem~\ref{thm:equivalence},  and therefore, $(D,\dff_D)\simeq (A,\dff_A) \cong (C,\dff_C)\simeq (D',\dff_{D'})$, 
which proves $(D,\dff_D)\simeq (D',\dff_{D'})$, i.e. they are $\sO(\sP)$-chain equivalent.
\qed
\end{proof}
\begin{remark}
    Conjecture 8.4 in \cite{robbin:salamon2} is false in the setting of connection matrices since in the case of ring coefficients the homologies $E^\beta_\alpha$, $\beta\smin\alpha=\{p\}$ are not necessarily free, or projective, cf.\ \cite[Exam.\ 6.3]{fran}. Therefore a connection matrix is not an option. 
    This problem is a subject of further study.
\end{remark}

\subsection{Connection matrix theories for vector spaces}
\label{connvectsp}
In \cite{hms} the theory of connection matrices for field coefficients is treated in the setting of homotopy categories without using homology.
For an $\sO(\sP)$-filtered differential vector space $(D,\dff_D)$, or equivalently, a $\sP$-graded differential vector space consider the decomposition
\[
D=\bigoplus_{p\in \sP} G_p D,\quad \alpha \mapsto F_\alpha D= \bigoplus_{p\in \alpha} G_p D,\quad\alpha\in \sO(\sP).
\]
In \cite[Defn.\ 4.23]{hms}, a strict, $\sP$-graded differential vector space $(A,\dff_A)$ is a connection matrix for $(D,\dff_D)$ if $(A,\dff_A)$ is $\sO(\sP)$-filtered chain equivalent to $(D,\dff_D)$. 
An algorithm is proven in   \cite{hms} to construct a $\sP$-graded differential vector space $(A,\dff_A)$ given a
finite dimensional $\sO(\sP)$-filtered differential vector space $(D,\dff_D)$.
This procedure is reminiscent of the inductive existence result in \cite[Thm.\ 8.1]{robbin:salamon2}.
In \cite[Prop.\ 4.27]{hms} it is proven that connection matrices defined in the sense of
\cite[Defn.\ 4.23]{hms} are $\sO(\sP)$-filtered chain isomorphic, which covers a special case of Theorem \ref{thm:equivalence}.

In Section \ref{franmodbr} we have summarized the connection matrix theory by Franzosa. Consider the $\sO(\sP)$-filtered differential vector space $(D,\dff_D)$.
In Franzosa's theory a strict, $\sP$-graded differential vector space $(A,\dff_A)$ is a connection matrix for an $\sO(\sP)$-filtered differential vector space $(D,\dff_D)$ if 
 $\bfE(A,\dff_A)\cong \bfE(D,\dff_D)$, cf.\ Defn.\ \ref{connmat2}.
 Existence of a connection matrix is given by \cite[Thm.\ 4.8]{fran}.
 The latter is not assumed to be a chain equivalence. Note that Franzosa's notion of connection matrix not even requires a map at the chain level.
 The following result concerns all possible connection matrices according to Definition \ref{connmat2}.

\begin{theorem}
    \label{equivKandF}
    Let $(D,\dff_D)$ be an $\sO(\sP)$-filtered differential vector space.
    Then, every strict, $\sP$-graded differential vector space $(A,\dff_A)$ which satisfies
    \[
    \bfE(A,\dff_A)\cong \bfE(D,\dff_D),
    \]
    is $\sO(\sP)$-filtered chain equivalent to $(D,\dff_D)$, i.e. $(A,\dff_A)\simeq (D,\dff_D)$.
\end{theorem}
\begin{proof}
   The $\sP$-graded differential vector space $(A,\dff_A)$ is also $\sO(\sP)$-filtered.
The $\sO(\sP)$-filtered differential vector spaces $(D,\dff_D)$ and $(A,\dff_A)$ satisfy
$\bfE(D,\dff_D) \cong \bfE(A,\dff_A)$. By Corollary \ref{conj74} this implies that 
$(D,\dff_D)$ and $(A,\dff_A)$ are $\sO(\sP)$-filtered chain equivalent, which proves the theorem.
    \qed
\end{proof}

As a consequence of Theorem \ref{equivKandF} the connection matrix theories of \cite{fran}, \cite{robbin:salamon2} and \cite{hms} in the setting of (finite dimensional) differential vector spaces are equivalent. The advantage of the latter theory is that the existence theory is based on an effective algorithm for producing connection matrices.

\subsection{Morse-Smale gradings and unique connection matrices}
\label{MSunique}
Suppose $(C,\dff_C)$ is a chain complex, in which case the chain homomorphisms are given by $f_k\colon C_k\to C_k$, i.e. $f_k(F_\alpha C_k)\subset F_\alpha C_k$, for all $\alpha\in \sO(\sP)$ and for all $k$, cf.\ Ex.\ \ref{chaincomplexex1} and Rmk.\ \ref{chaincomplexadd}.
For simplicity of exposition we consider
$\sP$-graded chain complexes in the category of $\Z_2$-vector spaces, where both gradings satisfy a natural order-preserving property.
\begin{definition}
    \label{MScc}
    A strict, $\sP$-graded chain complex $(C,\dff_C)$ is referred to as an \emph{algebraic Morse-Smale graded chain complex}
    if there exists a  map $\mu\colon \sP \to \Z$ such that
    \begin{enumerate}
        \item [(i)] $p<q$ implies $\mu(p)<\mu(q)$;
        \item[(ii)] $G_p C_k= \begin{cases}
     \Z_2 & \text{if }  k=\mu(p); \\
     0 & \text{otherwise}.
\end{cases}$
    \end{enumerate}
The pair $(\sP,\mu)$ is called a \emph{Morse-Smale grading} 
 for $(C,\dff_C)$. 
\end{definition}

By definition algebraic Morse-Smale graded chain complexes are finitely generated and $C= \Z_2^{|\sP|}$, where $|\sP|$ is the number of elements in $\sP$. As a consequence of   Condition (i)  we have that
\begin{equation}
    \label{consmu}
    \mu(p)=\mu(q),~~ p\neq q\quad \implies \quad p~\Vert~ q.
\end{equation}
The Morse-Smale grading restricts the differentials on $C$. 
\begin{theorem}
    \label{mainuniquethm}
    Let $(C,\dff_C)$ and $(C,\dff'_C)$ be   algebraic Morse-Smale graded chain complexes with Morse-Smale grading $(\sP,\mu)$ for both chain complexes.
    Then,  $\dff_C=\dff'_C$ if and only if $\bfE(C,\dff_C)\cong \bfE(C,\dff'_C)$.
\end{theorem}
\begin{proof}
    By assumption both chain complexes have the same Morse-Smale grading with associated choice of basis.
     Theorem \ref{thm:equivalence} yields the existence of an $\sO(\sP)$-filtered chain automorphism $f\colon C\to C$ given by $f_k\colon C_k\to C_k$, which are $\sO(\sP)$-filtered for all $k$. By the $\sP$-grading each $f_k$ may be regarded as matrix with entries
     $f_k^{pq}\colon G_q C_k\to G_p C_k$, for all $\mu(p)=\mu(q)=k$.
     The filtering condition on $f_k$ translates to: $f_k^{pq}\neq 0$ implies $p\le q$.
     Consequently, by \eqref{consmu}, and thus Definition \ref{MScc}(i),  $f_k^{pq}=0$ for all
     $p\neq q$ and $f_k^{pp}$ are automorphisms. By Definition \ref{MScc}(ii) the factors are given by $G_p C_k=\Z_2$ and therefore $f_k^{pp}=\id$ for all $p$ with $\mu(p)=k$ since the only group automorphism on $\Z_2$ is the identify map. We conclude that $f=\id$ and thus $\dff_C=\dff'_C$.
      \qed
\end{proof}

The term Morse-Smale graded chain complex is motivated by Morse-Smale flows in dynamical systems. In \cite{Reineck} J. Reineck proves that connection matrices for Morse-Smale flows without periodic orbits are unique. Via Conley index theory such flows yield a Morse-Smale graded chain complex where the groups $G_pC_k$ record the critical points of index $k$, cf.\ \cite{fran2}, \cite{Sal1}. Theorem \ref{mainuniquethm} provides an algebraic proof of Reineck's result.

\begin{remark}
The proof of Theorem \ref{mainuniquethm} can be repeated for chain complexes of $R$-modules. For example for chain complexes of abelian groups ($\Z$-modules) the result stays the same except
for the group automorphisms of $\Z$. The only options are $\pm \id$.
 This implies that the entries in $\dff_C$ and $\dff'_C$ may defer by a $\pm$ sign. This is exactly the choice of the orientation of the basis for the non-trivial entries of $\dff_C$.
 We conclude that the entries of the differentials  $\dff_C$ and $\dff'_C$ are the same up to conjugacy, which is a statement of uniqueness.
\end{remark}

\begin{remark}
    Theorem \ref{mainuniquethm} can be proved in a more general context when $(C,\dff_C)$ is an algebraic Morse-Smale graded chain complex and $(A,\dff_A)$ is a  free, finitely generated, strict $\sP$-graded chain complex of $R$-modules. Under the condition that $\bfE(C,\dff_C)\cong \bfE(C,\dff'_C)$, the modules $G_pC_k$ and $G_p A_k$ are isomorphic and 
    and the  entries in $\dff_C$ and $\dff_A$ are related by conjugation. This is the general statement that there exists only one strict differential to represent $\bfE$ in the Morse-Smale setting.
\end{remark}

\begin{remark}
The condition that the $\sO(\sP)$-filtered chain isomorphisms are also homogeneous with respect to the chain complex grading is crucial for Theorem \ref{mainuniquethm}.
If the latter is not satisfied, i.e. $f$  is $\sO(\sP)$-filtered, but necessarily homogeneous,
may yield multiple similar differentials.
\end{remark}

\begin{acknowledgements}
The first author was partially supported on NSF and NWO GROW grant nr.\ 040.15.044/3192, corr.\ nr.\ 2018/SGW00579061.
\end{acknowledgements}

\bibliographystyle{spmpsci}      % mathematics and physical sciences
\bibliography{references}   % name your BibTeX data base

\end{sloppypar}
\end{document}